\documentclass[11pt, reqno]{amsart}
\usepackage{hyperref}
\usepackage{graphicx}
\usepackage{color}%
\usepackage[numbers, square]{natbib}
\usepackage{mathrsfs}
\usepackage{mathptmx}
\usepackage{amsmath,amsthm,amsfonts,amssymb}
\usepackage{dsfont,extarrows}
\usepackage{enumitem}
\usepackage{fullpage}
\usepackage{xcolor}
\usepackage[normalem]{ulem}

\hypersetup{
  hidelinks,
  pdftitle={Exponential profiles for the intrinsic volumes of Orlicz balls},
  pdfsubject={Exact factorized integral representation, Laplace-transform asymptotics, and exact-mean inversion}
}

\allowdisplaybreaks[4]
\numberwithin{equation}{section}

\theoremstyle{plain} \newtheorem{theorem}{Theorem}[section]
\theoremstyle{plain} \newtheorem{proposition}[theorem]{Proposition}
\theoremstyle{plain} \newtheorem{lemma}[theorem]{Lemma}
\theoremstyle{plain} \newtheorem{corollary}[theorem]{Corollary}
\theoremstyle{remark} \newtheorem{remark}[theorem]{Remark}

\newcommand{\R}{\mathbb R}
\newcommand{\Vol}{\operatorname{Vol}}
\newcommand{\Var}{\operatorname{Var}}
\newcommand{\E}{\mathbb E}
\newcommand{\dd}{\,\mathrm d}
\newcommand{\cH}{\mathcal H}
\newcommand{\PhiA}{\Phi_{\alpha}}
\newcommand{\phiA}{\phi_{\alpha}}
\newcommand{\rate}{\mathcal I_M}

\title[Intrinsic volumes of Orlicz balls]{Exponential profiles for the intrinsic volumes of Orlicz balls}
\author{Mykyta Sadok}
\address{Faculty of Computer Science and Cybernetics, Taras Shevchenko National University of Kyiv, Kyiv, Ukraine}
\email{sadokmykyta@knu.ua}

\begin{document}

\begin{abstract}
An Orlicz ball is the set of points where the sum of a fixed even convex
function of the coordinates does not exceed a given radius, here
proportional to the dimension.  For a broad class of such functions,
including non-homogeneous ones, we find the exponential rate of the
intrinsic volumes whose order is proportional to the dimension and
express it by a two-parameter variational formula involving only
one-dimensional integrals.  As a function of the order-to-dimension
ratio, this rate is concave and continuous on the closed unit interval,
and the convergence is uniform in the order.  This proves the
exponential-profile conjecture of Kabluchko, Marynych, and Prochno for
these bodies and recovers their result for $\ell_p$ balls with $p>1$.
The proof relies on the theory of large deviations.  An exact formula for the Laplace transform of the intrinsic volumes in the radius, extending theirs for $\ell_p$ balls, plays the role of a moment generating
function, and a Laplace method gives its exponential rate.  An exponential change of measure, analogous to that used in Cram\'{e}r's theorem, then recovers the intrinsic volumes. The key step is a lower bound on the tilted density at its mean, obtained from log-concavity rather than a law of large numbers.
\end{abstract}



\maketitle

\section{Introduction}

The intrinsic volumes $V_0(K),\ldots,V_n(K)$ of a convex body
$K\subseteq\R^n$ are, up to normalization, the coefficients of the Steiner
polynomial $t\mapsto\Vol(K+tB_2^n)$, where $B_2^n$ is the Euclidean unit
ball.  They include the volume $V_n$, half
the surface area $V_{n-1}$, and a multiple of the mean width $V_1$, and by
Kubota's formula $V_j(K)$ is proportional to the average $j$-dimensional
volume of the projections of $K$ onto random $j$-dimensional subspaces
\cite{Schneider}.  For a sequence of convex bodies $K_n\subseteq\R^n$, the
orders $j_n$ with $j_n/n\to\alpha\in(0,1)$ form the bulk of the sequence
$j\mapsto V_j(K_n)$, and in many examples $V_{j_n}(K_n)$ grows or decays
exponentially in $n$ at a rate depending on $\alpha$.  Kabluchko, Marynych,
and Prochno~\cite{KabluchkoMarynychProchno} determined this rate for
$\ell_p$ balls and conjectured that every sufficiently regular sequence of
convex bodies with convergent $n^{-1}\log\Vol(K_n)$ admits a concave
exponential profile
\begin{equation*}
g(\alpha)=\lim_{n\to\infty}\frac1n
\log V_{\lfloor\alpha n\rfloor}(K_n),
\qquad 0\le\alpha\le1,
\end{equation*}
see \cite[Conjecture~2.19]{KabluchkoMarynychProchno}.

In this paper we determine the exponential profile of the Orlicz balls
\[
K_n(c)=\left\{x\in\R^n:\sum_{i=1}^n M(x_i)\le c\right\},
\qquad c=n\rho,\quad \rho>0.
\]
For $M(s)=|s|^p$, $p>1$, these are the dilated $\ell_p$ balls
$K_n(n\rho)=(n\rho)^{1/p}B_p^n$, where
\[
B_p^n=\left\{x\in\R^n:\sum_{i=1}^n|x_i|^p\le1\right\}
\]
is the unit $\ell_p$ ball.

Our main result, Theorem~\ref{thm:main}, applies to every $M$ satisfying
\textnormal{(M1)--(M2)} below.  Besides $|s|^p$ for all $p>1$, these
assumptions cover non-homogeneous functions such as $|s|^p+|s|^q$,
$\cosh s-1$, $e^{s^2}-1$, and $\sqrt{1+s^2}-1$, as well as convex
functions that are affine outside a bounded interval.  If the orders
$j_n\in\{0,\ldots,n-1\}$ satisfy $j_n/n\to\alpha\in(0,1)$, then
\[
\lim_{n\to\infty}\frac1n\log V_{j_n}(K_n(n\rho))
=H(\alpha)-\frac{1-\alpha}{2}\log\pi
 +\inf_{\lambda>0}\{\lambda\rho+\Phi_\alpha(\lambda)\},
\]
where $H(\alpha)$ is the entropy and $\Phi_\alpha(\lambda)$ is the supremum
over $\theta>0$ of an explicit combination of one-dimensional integrals,
see Section~\ref{sec:variational}.  Both the inner supremum and the outer
infimum have unique optimizers, and the optimal pair $(\theta,\lambda)$
is the unique solution of the two scalar equations
\eqref{eq-optimality-main} stated with the theorem.  We write
$\rate(\alpha,\rho)$ for the right-hand side.

Corollary~\ref{cor:profile} adds the two endpoints.  Every sublinear
sequence of orders has zero exponential rate, and at $\alpha=1$ the same
inversion recovers the exponential volume asymptotics of Orlicz balls.
The corollary also upgrades the convergence to a uniform one over all
orders $j\in\{0,\ldots,n\}$, which establishes the conjectured concave
profile for $K_n(n\rho)$ on all of $[0,1]$.  For $M(s)=|s|^p$,
Corollary~\ref{cor:profile} recovers the exponential profile of
\cite[Theorem~2.17]{KabluchkoMarynychProchno}, uniformly in the order,
and Proposition~\ref{prop:product} recovers the exact formula
\cite[Theorem~2.1]{KabluchkoMarynychProchno}, see
Section~\ref{sec:known-cases}.

The main difference from the $\ell_p$ case is the lack of homogeneity.
We replace it by the Laplace transform
$Q_{n,j}(\lambda)=\int_0^\infty e^{-\lambda c}V_j(K_n(c))\dd c$ in the
radius $c$.  Coarea turns the level-set integrals into an integral over
$\R^n$, and a Gamma identity makes the integrand factorize over the
coordinates (Proposition~\ref{prop:product}).  Endpoint estimates in the
variable of the Gamma identity then identify the exponential rate of
$Q_{n,j_n}(\lambda)$ as
\[
\Lambda_\alpha(\lambda)
=H(\alpha)-\frac{1-\alpha}{2}\log\pi+\Phi_\alpha(\lambda),
\]
and we write $\Lambda=\Lambda_\alpha$ when $\alpha$ is fixed.  The
heuristic behind the inversion is the following.  If
$V_j(K_n(c))\approx e^{n\rate(\alpha,c/n)}$, then
\[
Q_{n,j}(\lambda)
\approx\int_0^\infty e^{\,n(\rate(\alpha,c/n)-\lambda c/n)}\dd c,
\]
so $\Lambda$ should be the Legendre transform
$\sup_{\rho>0}\{\rate(\alpha,\rho)-\lambda\rho\}$ of the profile, and
concavity of the profile in $\rho$ should return
$\rate(\alpha,\rho)=\inf_{\lambda>0}\{\lambda\rho+\Lambda(\lambda)\}$.
We make this exact without a Tauberian argument, by an exponential
change of measure.  Minimizing the finite-$n$ transform objective
$\lambda\rho+n^{-1}\log Q_{n,j}(\lambda)$ selects the tilted probability
density proportional to $e^{-\lambda c}V_j(K_n(c))$ whose mean is
exactly $n\rho$, and log-concavity of $c\mapsto V_j(K_n(c))$, a
consequence of Brunn--Minkowski, bounds this density at its mean from
below by a polynomial factor.  Remark~\ref{rem:radius} proves the two
duality identities.

The main asymptotic tools are thus the exact transform representation,
the Laplace principle, and an exponential change of measure.  Although
the main results concern deterministic geometric quantities, the last
two follow the pattern of large-deviation theory \cite{DemboZeitouni}.
The Laplace principle, the deterministic prototype of Varadhan's
integral lemma, identifies the exponential rate of the $\theta$-integral
in \eqref{eq:product} as the supremum of its exponent, which produces
the variational formula for $\Phi_\alpha$
(Proposition~\ref{prop:transform-rate}).  The inversion
(Proposition~\ref{prop:inversion}) uses the exponential change of measure
from the proof of the lower bound in Cram\'er's theorem, where the tilt
makes the target value typical, with log-concavity in place of the law
of large numbers used there.  In this analogy, $\Lambda$ plays the role
of a limiting scaled cumulant generating function and
$-\rate(\alpha,\cdot)$ that of a rate function, and the two duality
identities correspond to the Legendre duality in the G\"artner--Ellis
theorem, obtained here for the values $V_{j_n}(K_n(n\rho))$ themselves
rather than for probabilities.  Remark~\ref{rem:radius} adds the
corresponding free-energy interpretation.

High-dimensional volume asymptotics and coordinate behavior for Orlicz
balls were studied in \cite{BartheWolff,KabluchkoProchno}.  The classical
Maxwell principle for $\ell_p$ spheres appears in \cite{NaorRomik}, with
extensions to Orlicz balls in
\cite{FruehwirthProchno,JohnstonProchno}.  Related Gibbs-measure methods and
thin-shell concentration for Orlicz balls were developed in
\cite{AlonsoGutierrezProchno}.  For $\ell_p$ balls, exact one-dimensional
integral representations, proportional-order intrinsic-volume asymptotics,
and Maxwell--Poincar\'e--Borel laws for curvature measures were obtained in
\cite{KabluchkoMarynychProchno}.  Remark~\ref{rem:optimality} indicates
how the last of these appears in our variational formula.  The
factorization of Proposition~\ref{prop:product} persists for generalized
Orlicz balls with coordinate-dependent functions $M_1,\ldots,M_n$ as in
\cite{JohnstonProchno}, but the sum over coordinate subsets then no longer
collapses to a binomial coefficient, see Section~\ref{sec-open}.

\section{Setup and main results}\label{sec:framework}

\subsection{Notation and preliminaries}\label{sec:prelim}

Our standard references are Schneider~\cite{Schneider} for convex
geometry, Rockafellar~\cite{Rockafellar} for convex analysis, and Evans
and Gariepy~\cite{EvansGariepy} for the coarea formula.

\emph{General notation.}
$\langle x,y\rangle$ and $\lVert x\rVert$ are the Euclidean inner product
and norm, $B_2^n$ is the closed Euclidean unit ball, $\R_+=[0,\infty)$,
$\Vol$ is Lebesgue measure on $\R^n$, and $\cH^{n-1}$ is
$(n-1)$-dimensional Hausdorff measure, normalized so that on a smooth
hypersurface it is surface area.  For a probability measure $\mu$,
$\E_\mu[f]=\int f\dd\mu$ and
$\Var_\mu(f)=\E_\mu[f^2]-\E_\mu[f]^2$.
For
$0\le r\le N$,
\[
e_r(z_1,\ldots,z_N)
=\sum_{\substack{I\subseteq\{1,\ldots,N\}\\|I|=r}}
\prod_{i\in I}z_i,
\qquad e_0=1,
\]
is the $r$-th elementary symmetric polynomial.

\emph{Convex bodies.}
A \emph{convex body} is a compact convex subset of $\R^n$ with nonempty
interior.  For $y\in K$, the normal cone is
$N_K(y)=\{w:\langle w,z-y\rangle\le0\text{ for every }z\in K\}$, and the
metric projection $\Pi_K(x)$, the unique nearest point of $K$ to $x$, is
characterized by $x-\Pi_K(x)\in N_K(\Pi_K(x))$, see
\cite[Section~1.2]{Schneider}.  If $\partial K$ is a $C^1$ hypersurface
with outward unit normal $\nu$, then the supporting hyperplane at each
boundary point is unique and $N_K(y)=\R_+\nu(y)$, see
\cite[Section~2.2]{Schneider}.  On a relatively open $C^2$ part of
$\partial K$, the shape operator
$S_x=D\nu(x)|_{T_x\partial K}$ (convex-geometry sign convention, no
minus sign) is self-adjoint with nonnegative eigenvalues
$\chi_1,\ldots,\chi_{n-1}$, the \emph{principal curvatures}, see
\cite[Section~2.5]{Schneider}.

\emph{Coarea formula.}
If $F$ is Lipschitz on a bounded open set $U\subseteq\R^n$ and $\psi\ge0$
is measurable on $U$, then
\[
\int_U\psi(x)\lVert\nabla F(x)\rVert\dd x
=\int_\R\int_{U\cap F^{-1}(c)}\psi\dd\cH^{n-1}\dd c,
\]
see \cite[Section~3.4]{EvansGariepy}.

\subsection{Defining functions and Orlicz balls}\label{sec:defining}

Let $M\colon\R\to\R$, the \emph{defining function}, satisfy the
assumptions below.
\begin{enumerate}[label=\textnormal{(M\arabic*)}]
\item $M$ is even and convex, with $M(0)=0$ and $M(s)>0$ for $s\ne0$.
\item $M\in C^1(\R)\cap C^2(\R\setminus\{0\})$.
\end{enumerate}
Under \textnormal{(M1)}, the quotient $M(s)/s$ is nondecreasing and
positive on $(0,\infty)$, by convexity and $M(0)=0$.  Hence
$M(s_1)\le(s_1/s_2)M(s_2)<M(s_2)$ for $0<s_1<s_2$, so $M$ is strictly
increasing on $[0,\infty)$, and for every fixed $a>0$,
\[
M(s)\ge \frac{M(a)}a|s|\qquad(|s|\ge a),
\]
and in particular $M(s)\to\infty$ at least linearly.  Convexity and
\textnormal{(M2)} give $M''\ge0$ away from the origin. No positive lower
bound is assumed, and no particular behavior near the origin is required.
If $M''(0)$ is undefined, as happens for $|s|^p$ with $1<p<2$, we assign
it the value zero.  This convention is immaterial, since $M''$ enters only
through one-dimensional integrals, for which the single point $0$ is
null, and through the curvature formulas of
Section~\ref{sec:factorization}, which are used only off the coordinate
hyperplanes.

The derivative $M'$ is odd, continuous, and nondecreasing, and the
supporting-line inequality $sM'(s)\ge M(s)$ gives $sM'(s)>0$ for $s\ne0$.
Its range is an interval, open or closed, with endpoints $\pm L$, where
\[
L=\lim_{s\to\infty}M'(s)\in(0,\infty].
\]
Strict convexity is not assumed, so $M'$ may be constant on intervals.
For $r\in(-L,L)$, define $\sigma(r)$ to be the midpoint of the compact
interval $\{s\in\R:M'(s)=r\}$.  Then $\sigma$ is odd, nondecreasing, and
agrees with $(M')^{-1}$ whenever the inverse is single-valued.  Since
$M'$ is monotone, at most countably many of its level sets are
nondegenerate intervals.  The corresponding values of $r$ form a null
set, so the choice of $\sigma$ there does not affect the integrals in $r$
below.  The assumptions also allow asymptotically linear defining functions,
for which $L<\infty$, including functions that are affine outside a
bounded interval.

Evenness is part of the classical definition of an Orlicz function, but
it is not essential here, provided it is replaced by the requirement that
$0$ minimizes $M$.  It is used only to write $M'(0)=0$, which then still
holds, to have the symmetric interior $(-L,L)$ of the derivative range,
and to work with a single inverse $v$ below.  When $L<\infty$, the
endpoints may also belong to the range, as happens for affine tails.
Without evenness, the interior of the derivative range is $(-L_-,L_+)$
with $L_\pm\in(0,\infty]$.  The two one-sided inverses of $M$ enter
\eqref{eq:body-sandwich} separately, and symmetric bounds, such as those
in Lemma~\ref{lem:coercive} and Corollary~\ref{cor:profile}, are replaced
by their one-sided versions.
We keep evenness for notational simplicity.

No growth condition on $M''$ at infinity is imposed.  The assumptions
are satisfied by
$M(s)=|s|^p$ with $p>1$ (with a singular second derivative at zero when
$1<p<2$), by $|s|^p+|s|^q$ with $1<p<q$, and by $\cosh s-1$ and
$e^{s^2}-1$.  They also include the asymptotically linear example
$M(s)=\sqrt{1+s^2}-1$ and the Huber-type function with
$M(s)=|s|^3-\tfrac12|s|^4$ for $|s|\le1$ and $M(s)=|s|-\tfrac12$ for
$|s|\ge1$, which is $C^2$, affine for $|s|\ge1$, and not strictly convex.

For $c>0$, write
\[
F_n(x)=\sum_{i=1}^nM(x_i),\qquad K_n(c)=\{x\in\R^n:F_n(x)\le c\}.
\]
Let $v=(M|_{[0,\infty)})^{-1}$.  Monotonicity of $M$ on $[0,\infty)$ gives
\begin{equation}\label{eq:body-sandwich}
v(c/n)B_2^n\subset K_n(c)\subset \sqrt n\,v(c)B_2^n.
\end{equation}
Indeed, the first inclusion follows from
$|x_i|\le\lVert x\rVert\le v(c/n)$, and the second from
$M(x_i)\le c$.  These containments show that $K_n(c)$ is a convex body,
and the linear lower growth of $M$ gives $v(c)\le C(1+c)$.
We use $c$ as the radius parameter and $\rho=c/n$ as its normalized value.

\subsection{Intrinsic volumes}

Intrinsic volumes are defined by the following Steiner formula.
\[
\Vol(K+tB_2^n)=\sum_{j=0}^n\kappa_{n-j}V_j(K)t^{n-j},
\qquad t\ge0,\qquad
\kappa_r=\Vol(B_2^r)=\frac{\pi^{r/2}}{\Gamma(r/2+1)}.
\]
Here
\[
K+tB_2^n=\{x+ty:x\in K,\ y\in B_2^n\}
=\{z:\operatorname{dist}(z,K)\le t\}
\]
is the outer parallel body of $K$ at distance $t$, and $\kappa_0=1$.  This
normalization gives
\[
V_n(K)=\Vol(K),\qquad V_0(K)=1,\qquad
2V_{n-1}(K)=\cH^{n-1}(\partial K).
\]
For convex bodies $K$ and $K'$ we use the following three properties.
The first two, monotonicity and homogeneity, follow from the
corresponding properties of mixed volumes \cite[Section~5.1]{Schneider}.
The third is the Brunn--Minkowski inequality for intrinsic volumes, the
case of the general Brunn--Minkowski theorem for mixed volumes
\cite[Section~7.4]{Schneider} in which all but $j$ arguments are
Euclidean balls.  Here $(1-t)K+tK'$ denotes the set of all $(1-t)x+ty$
with $x\in K$ and $y\in K'$.
\begin{align}
K\subset K'&\Longrightarrow V_j(K)\le V_j(K'),
\nonumber\\
V_j(aK)&=a^jV_j(K),\qquad a>0,
\nonumber\\
V_j((1-t)K+tK')^{1/j}
&\ge(1-t)V_j(K)^{1/j}+tV_j(K')^{1/j},
\quad 1\le j\le n,\quad 0\le t\le1.
\label{eq:intrinsic-BM}
\end{align}

\subsection{Variational exponent}\label{sec:variational}

For $\theta,\lambda\ge0$, define, with values in $(0,\infty]$,
\begin{align*}
Z(\theta,\lambda)&=\int_\R e^{-\lambda M(s)-\theta M'(s)^2}\dd s,\\*
B(\theta,\lambda)&=\int_\R M''(s)e^{-\lambda M(s)-\theta M'(s)^2}\dd s,\\*
A(\theta,\lambda)&=\int_\R M'(s)^2e^{-\lambda M(s)-\theta M'(s)^2}\dd s.
\end{align*}
Lemma~\ref{lem:endpoints} shows that all three are finite for
$\theta,\lambda>0$, and $Z$ also for $\theta=0$.
Values at $\lambda=0$, where $Z$ may be infinite, are needed only in
Lemma~\ref{lem:endpoints} and in the proof of Lemma~\ref{lem:coercive}.
For $\alpha\in(0,1)$, put
\begin{equation}\label{eq:variational-exponent}
\begin{aligned}
\phiA(\theta,\lambda)
&=\frac{1-\alpha}{2}\log\theta
 +\alpha\log Z(\theta,\lambda)
 +(1-\alpha)\log B(\theta,\lambda),\\
\PhiA(\lambda)&=\sup_{\theta>0}\phiA(\theta,\lambda),
\end{aligned}
\end{equation}
and set $H(\alpha)=-\alpha\log\alpha-(1-\alpha)\log(1-\alpha)$, with the
convention $0\log0=0$ at the endpoints.  All logarithms are natural.
Dependence on $M$ is suppressed in the notation.

The parameter $\lambda$ is the variable of the Laplace transform in the
radius, while $\theta$ comes from a Gamma representation of a negative
power of the gradient norm in the curvature formula.  The powers of $Z$ and
$B$ determine the variational exponent.  The function $A$ enters only the exact formula of Section~\ref{sec:factorization}.  The term $H(\alpha)$ counts the choices of curvature
coordinates, and $-\frac{1-\alpha}{2}\log\pi$ comes from the volumes of
Euclidean balls in the normalization of intrinsic volumes.  The derivation
appears in Section~\ref{sec:factorization}.  Increasing $\theta$ decreases
$Z(\theta,\lambda)$ while increasing
$\sqrt\theta\,B(\theta,\lambda)$; the inner optimizer balances these two
effects with weights $\alpha$ and $1-\alpha$.

\subsection{Main result}

\begin{theorem}[Exponential profile]\label{thm:main}
Assume \textnormal{(M1)--(M2)}.  If $j_n\in\{0,\ldots,n-1\}$ and
$j_n/n\to\alpha\in(0,1)$, then, for every fixed $\rho>0$,
\begin{equation}\label{eq:mainrho}
\lim_{n\to\infty}\frac1n\log V_{j_n}\bigl(K_n(n\rho)\bigr)
=H(\alpha)-\frac{1-\alpha}{2}\log\pi
 +\inf_{\lambda>0}\{\lambda\rho+\PhiA(\lambda)\}.
\end{equation}
For every $\lambda>0$, the inner supremum defining $\PhiA(\lambda)$ has a
unique maximizer $\theta_*(\lambda)$, and for every $\rho>0$ the outer
infimum has a unique minimizer $\lambda_\rho$.  The pair
$(\theta,\lambda)=(\theta_*(\lambda_\rho),\lambda_\rho)$ is the unique
solution in $(0,\infty)^2$ of the two equations
\begin{equation}\label{eq-optimality-main}
\begin{gathered}
2\alpha\theta B(\theta,\lambda)=(1-\alpha)\lambda Z(\theta,\lambda),\\
\alpha\,\frac{\int_\R M(s)e^{-\lambda M(s)-\theta M'(s)^2}\dd s}
{Z(\theta,\lambda)}
+(1-\alpha)\,\frac{\int_\R M(s)M''(s)e^{-\lambda M(s)-\theta M'(s)^2}\dd s}
{B(\theta,\lambda)}=\rho.
\end{gathered}
\end{equation}
Setting $\rho=1$ gives the profile of $K_n(n)$.
\end{theorem}

Write $\rate(\alpha,\rho)$ for the right-hand side of
\eqref{eq:mainrho}.  Equivalently,
$V_{j_n}(K_n(n\rho))=\exp\{n\rate(\alpha,\rho)+o(n)\}$ as $n\to\infty$.
The two quotients in \eqref{eq-optimality-main} are the means of $M$
under the probability densities proportional to
$e^{-\lambda M-\theta M'^2}$ and to $M''e^{-\lambda M-\theta M'^2}$, so
the second equation matches the normalized radius $\rho$ to a mixture,
with weights $\alpha$ and $1-\alpha$, of two coordinate laws.
Section~\ref{sec:known-cases} solves \eqref{eq-optimality-main} in
closed form for $M(s)=s^2$.

The full profile is given by the following corollary, proved in
Section~\ref{sec:inversion}.

\begin{corollary}[{The profile on $[0,1]$}]\label{cor:profile}
Assume \textnormal{(M1)--(M2)} and fix $\rho>0$.  Put $\rate(0,\rho)=0$ and
\[
\rate(1,\rho)=\inf_{\lambda>0}\{\lambda\rho+\log Z(0,\lambda)\}.
\]
\begin{enumerate}[label=\textnormal{(\roman*)},nosep]
\item The infimum has a unique minimizer, characterized by
$\int_\R M(s)e^{-\lambda M(s)}\dd s=\rho Z(0,\lambda)$, and
\[
\lim_{n\to\infty}\frac1n\log\Vol(K_n(n\rho))=\rate(1,\rho).
\]
\item The function $\alpha\mapsto\rate(\alpha,\rho)$ is concave and
continuous on $[0,1]$, and
\begin{equation}\label{eq:full-uniform}
\lim_{n\to\infty}\max_{0\le j\le n}
\Bigl|\frac1n\log V_j(K_n(n\rho))-\rate\bigl(\tfrac jn,\rho\bigr)\Bigr|=0.
\end{equation}
In particular, $n^{-1}\log V_{j_n}(K_n(n\rho))\to\rate(\alpha,\rho)$ for
every sequence $j_n\in\{0,\ldots,n\}$ with $j_n/n\to\alpha\in[0,1]$.
Thus every sublinear sequence $j_n=o(n)$ has limit
$\rate(0,\rho)=0$, and $j_n=\lfloor\alpha n\rfloor$ has limit
$\rate(\alpha,\rho)$ for every $\alpha\in[0,1]$.
\end{enumerate}
\end{corollary}

For example, when $M(s)=s^2$, the formula gives
\[
\rate(\alpha,\rho)=H(\alpha)+\frac\alpha2\log(2\pi e\rho)
 +\frac{1-\alpha}{2}\log(1-\alpha).
\]
Section~\ref{sec:known-cases} derives this expression and checks it against
the exact intrinsic volumes of a Euclidean ball.

\subsection{Intrinsic volumes as functions of the radius}
\label{sec-radius-profiles}

The functions $c\mapsto V_j(K_n(c))$ enter the proof twice.  Their
polynomial growth makes the radius Laplace transform of
Section~\ref{sec:factorization} finite, and their log-concavity, which
follows from Brunn--Minkowski in the stronger form below, drives the
inversion of Section~\ref{sec:inversion}.

\begin{lemma}[Intrinsic-volume profiles]\label{lem:profile}
For $n\ge1$ and $1\le j\le n$, the function
$f_{n,j}(c)=V_j(K_n(c))$ is positive and nondecreasing on $(0,\infty)$,
has at most polynomial growth, and $f_{n,j}^{1/j}$ is concave.
In particular, $f_{n,j}$ is log-concave.
\end{lemma}

\begin{proof}
\emph{Positivity, monotonicity, and growth.}
The left inclusion in \eqref{eq:body-sandwich}, monotonicity, and
homogeneity give
\[
f_{n,j}(c)\ge v(c/n)^jV_j(B_2^n)>0.
\]
Also $c_0\le c_1$ implies $K_n(c_0)\subset K_n(c_1)$, so
$f_{n,j}(c_0)\le f_{n,j}(c_1)$.
The right inclusion in \eqref{eq:body-sandwich} and the bound
$v(c)\le C(1+c)$ yield
\[
f_{n,j}(c)
\le n^{j/2}v(c)^jV_j(B_2^n)
\le C_{n,j}(1+c)^j.
\]
The constant depends on $n$, $j$, and $M$, and no uniformity in $n$ is
needed.

\emph{Concavity.}
Let $c_t=(1-t)c_0+tc_1$ with $c_0,c_1>0$ and $0\le t\le1$.
For $x_0\in K_n(c_0)$ and $x_1\in K_n(c_1)$, convexity gives
\[
F_n((1-t)x_0+tx_1)
\le(1-t)F_n(x_0)+tF_n(x_1)
\le c_t.
\]
Consequently,
\[
(1-t)K_n(c_0)+tK_n(c_1)\subset K_n(c_t).
\]
Set $u=f_{n,j}^{1/j}$.  Applying monotonicity to this inclusion and then
\eqref{eq:intrinsic-BM}, we obtain
\begin{align*}
u(c_t)
&\ge V_j((1-t)K_n(c_0)+tK_n(c_1))^{1/j}\\
&\ge(1-t)V_j(K_n(c_0))^{1/j}+tV_j(K_n(c_1))^{1/j}\\
&=(1-t)u(c_0)+tu(c_1).
\end{align*}
Thus $u$ is positive and concave.  Weighted arithmetic--geometric mean
now gives
\[
u(c_t)\ge(1-t)u(c_0)+tu(c_1)
\ge u(c_0)^{1-t}u(c_1)^t.
\]
Raising this inequality to the power $j$ proves log-concavity of
$f_{n,j}$.
\end{proof}

Before turning to the proofs, we record where the regularity assumption
is used.

\begin{remark}[On the role of \textnormal{(M2)}]\label{rem-assumptions}
Assumption \textnormal{(M2)} enters in two ways.  First, through the
curvature formulas of Section~\ref{sec:factorization}.  The $C^2$
regularity of $M$ off the origin makes $\partial K_n(c)$ a $C^2$
hypersurface off the coordinate hyperplanes, which is the hypothesis of
Lemma~\ref{lem-curvature-density}, and it makes the curvature integrand
of Section~\ref{sec:coarea} continuous, hence measurable, on the
complement of the coordinate hyperplanes.  Second, it gives the
local absolute continuity of $M'$ established at the beginning of the
proof of Lemma~\ref{lem:endpoints}.  All remaining arguments, in
particular the one-dimensional estimates of Section~\ref{sec:asymptotics}
and the inversion of Section~\ref{sec:inversion}, use only that $M'$ is
odd, nondecreasing, and locally absolutely continuous with $M''\ge0$
almost everywhere.  We do not pursue the resulting relaxation of
\textnormal{(M2)} here.
\end{remark}

\section{The exact factorized integral formula}\label{sec:factorization}

Fix $n\ge2$ and $j\in\{0,\ldots,n-1\}$, and define the radius Laplace
transform
\[
Q_{n,j}(\lambda)=\int_0^\infty e^{-\lambda c}V_j(K_n(c))\dd c,
\qquad \lambda>0.
\]
For $j\ge1$, the polynomial growth bound in Lemma~\ref{lem:profile} makes
this integral finite.  For $j=0$, the identity $V_0(K_n(c))=1$ gives
$Q_{n,0}(\lambda)=1/\lambda$.  The main asymptotic theorem only uses
$1\le j_n\le n-1$, which holds eventually when
$j_n/n\to\alpha\in(0,1)$.  Including $j=0$ here records the full range of
the exact formula.

As explained in the introduction, the Laplace transform in $c$ replaces
the homogeneity used in \cite[Proposition~4.9]{KabluchkoMarynychProchno}
for $M(s)=|s|^p$.

\begin{proposition}[Exact factorized integral formula]\label{prop:product}
For every $n\ge2$, every $j\in\{0,\ldots,n-1\}$, and every $\lambda>0$,
\begin{equation}\label{eq:product}
\boxed{
Q_{n,j}(\lambda)
=\frac{n\binom{n-1}{j}}{2\pi^{(n-j)/2}}
\int_0^\infty
\theta^{(n-j)/2-1}Z(\theta,\lambda)^j
B(\theta,\lambda)^{n-1-j}A(\theta,\lambda)\dd\theta.}
\end{equation}
\end{proposition}

Throughout this section, $k=n-1-j$, so that $k+1=n-j$.
Section~\ref{sec:curvature} computes the principal curvatures off the
coordinate hyperplanes, Section~\ref{sec:curvature-measures} expresses
$V_j(K_n(c))$ as a curvature integral over this regular part of the
boundary, and Section~\ref{sec:coarea} applies the coarea formula and
separates the coordinates.

\subsection{Curvature on the regular boundary}\label{sec:curvature}

For $c>0$, write
\[
\Sigma_c=\partial K_n(c)=\{x\in\R^n:F_n(x)=c\},
\qquad
E=\bigcup_{i=1}^n\{x_i=0\},
\qquad
\Sigma_c^\circ=\Sigma_c\setminus E.
\]
Since $M'(s)\ne0$ for $s\ne0$ (Section~\ref{sec:defining}),
$\nabla F_n(x)=(M'(x_1),\ldots,M'(x_n))$ vanishes only at $x=0\notin\Sigma_c$,
so $\Sigma_c$ is a regular level set of the $C^1$ function $F_n$.  It is a
$C^1$ hypersurface with tangent space
$T_x\Sigma_c=\nabla F_n(x)^\perp$, and since $F_n\le c$ on $K_n(c)$, the
outward unit normal is $\nu=\nabla F_n/\lVert\nabla F_n\rVert$.  Off the
coordinate hyperplanes all summands of $F_n$ are $C^2$, so $\Sigma_c$ is
$C^2$ in a neighborhood of every point of the relatively open set
$\Sigma_c^\circ$, where we use the principal curvatures
$\chi_1,\ldots,\chi_{n-1}$ of Section~\ref{sec:prelim}.  Write
\[
g_i=M'(x_i),\qquad g=\nabla F_n(x),\qquad
h_i=M''(x_i),\qquad W=D^2F_n(x)=\operatorname{diag}(h_1,\ldots,h_n),
\]
the last two on $\Sigma_c^\circ$.  The following identity appears in
\cite[Proposition~4.2]{KabluchkoMarynychProchno}. We include its short
proof for the reader's convenience.

\begin{lemma}[Curvature identity]\label{lem:curvature}
At every point of $\Sigma_c^\circ$,
\begin{equation}\label{eq:curvature}
e_r(\chi_1,\ldots,\chi_{n-1})
=\frac1{\lVert g\rVert^{r+2}}
\sum_{i=1}^n g_i^2e_r\bigl((h_\ell)_{\ell\ne i}\bigr),
\qquad 0\le r\le n-1.
\end{equation}
\end{lemma}

\begin{proof}
Since $\nu=g/\lVert g\rVert$ and
$Dg=W$, the shape operator is $S_x=\lVert g\rVert^{-1}(PWP)|_{\nu^\perp}$,
where $P=\mathrm{Id}-\nu\nu^{\mathsf T}$ is the orthogonal projection onto
$\nu^\perp$.  Hence, for $u\ge0$, the product
$\prod_{\ell=1}^{n-1}(1+u\chi_\ell)$ is the determinant of the compression
to $\nu^\perp$ of the positive definite diagonal matrix
$Y=\mathrm{Id}+uW/\lVert g\rVert$.  In an orthonormal basis whose first vector is $\nu$, this compression
is the lower right block of $Y$, and the Schur complement formula for
$\det Y$ shows that its determinant equals
\[
\det(Y)\langle\nu,Y^{-1}\nu\rangle
=\lVert g\rVert^{-2}\sum_{i=1}^ng_i^2
\prod_{\ell\ne i}\Bigl(1+\frac{uh_\ell}{\lVert g\rVert}\Bigr),
\]
where the second expression uses that $Y$ is diagonal and
$\nu_i=g_i/\lVert g\rVert$.  Comparing the coefficients of $u^r$ on both sides gives
\eqref{eq:curvature}.
\end{proof}

\subsection{Curvature measures and intrinsic volumes}
\label{sec:curvature-measures}

For a convex body $K\subseteq\R^n$, a Borel set $\beta\subseteq\partial K$,
and $t\ge0$, the local parallel set is
\[
\mathcal M_t(K,\beta)=\{x\in(K+tB_2^n)\setminus K:\Pi_K(x)\in\beta\}.
\]
The curvature measures $\mathcal C_0(K,\cdot),\ldots,\mathcal C_{n-1}(K,\cdot)$
(denoted $\Phi_j$ in \cite{KabluchkoMarynychProchno,Schneider}) are the
finite Borel measures on $\partial K$ determined by the local Steiner
formula
\begin{equation}\label{eq:local-steiner}
\Vol(\mathcal M_t(K,\beta))=\sum_{j=0}^{n-1}\kappa_{n-j}t^{n-j}\mathcal C_j(K,\beta),
\qquad t\ge0,
\end{equation}
see \cite[Section~4.2]{Schneider}.  For $\beta=\partial K$ this is
Steiner's formula, so $\mathcal C_j(K,\partial K)=V_j(K)$.  On the $C^2$
part of the boundary, the curvature measures have densities given by the
elementary symmetric functions of the principal curvatures.  In
\cite{Schneider} this is stated for bodies of class $C^2_+$, that is,
with $C^2$ boundary and positive principal curvatures, see
\cite[Section~4.2, (4.25)]{Schneider} and
\cite[(3.1)]{KabluchkoMarynychProchno}.  We need a local version in which
the boundary is $C^2$ only near the set in question and principal
curvatures may vanish.

\begin{lemma}[Curvature-measure density on the $C^2$ part]
\label{lem-curvature-density}
Let $K\subseteq\R^n$ be a convex body and let $\Omega\subseteq\partial K$
be relatively open with $\partial K$ of class $C^2$ in a neighborhood of
every point of $\Omega$.  Then, for every Borel set $\beta\subseteq\Omega$,
\begin{equation}\label{eq:curvature-density}
\mathcal C_{n-1-r}(K,\beta)
=\frac1{(r+1)\kappa_{r+1}}\int_\beta e_r(\chi)\dd\cH^{n-1},
\qquad 0\le r\le n-1,
\end{equation}
where $\chi=(\chi_1,\ldots,\chi_{n-1})$ are the principal curvatures.
\end{lemma}

\begin{proof}
Both sides of \eqref{eq:curvature-density} are Borel measures on
$\Omega$.  The left side is finite.  The right side is finite on compact
subsets of $\Omega$, because $e_r(\chi)$ is continuous on $\Omega$ and
$\cH^{n-1}$ is finite on compact subsets of a $C^1$ hypersurface.  Being
relatively open in the compact set $\partial K$, the set $\Omega$ is a
$\sigma$-compact metric space, so both sides are Radon measures on
$\Omega$ and are determined by their values on compact sets by inner
regularity.  It therefore suffices to take $\beta$ compact.

On a neighborhood of $\beta$ the normal $\nu$ is $C^1$.  Consider the map
$\Psi(y,\tau)=y+\tau\nu(y)$ on $\beta\times(0,t]$.  Its image is
$\mathcal M_t(K,\beta)$, and it is injective, because
$\Pi_K(y+\tau\nu(y))=y$ for $\tau\ge0$ by the characterization of the
metric projection in Section~\ref{sec:prelim}.  It is Lipschitz, being
$C^1$ on a neighborhood of the compact set $\beta\times[0,t]$.  In an
orthonormal frame of principal directions at $y$, completed by $\nu(y)$,
its differential is
$\operatorname{diag}(1+\tau\chi_1(y),\ldots,1+\tau\chi_{n-1}(y),1)$, so
its Jacobian is $\prod_{\ell=1}^{n-1}(1+\tau\chi_\ell(y))$.  The area
formula \cite[Section~3.3]{EvansGariepy} gives
\[
\Vol(\mathcal M_t(K,\beta))
=\int_\beta\int_0^t\prod_{\ell=1}^{n-1}(1+\tau\chi_\ell(y))\dd\tau
\dd\cH^{n-1}(y)
=\sum_{r=0}^{n-1}\frac{t^{r+1}}{r+1}\int_\beta e_r(\chi)\dd\cH^{n-1}.
\]
In \eqref{eq:local-steiner}, the power $t^{r+1}$ carries the coefficient
$\kappa_{r+1}\mathcal C_{n-1-r}(K,\beta)$.  Comparing coefficients gives
\eqref{eq:curvature-density}.
\end{proof}

We apply this to $K=K_n(c)$ and $\Omega=\Sigma_c^\circ$.  The exceptional
set $\Sigma_c\cap E$ is not charged by any curvature measure.  Indeed,
let $x\in \mathcal M_t(K_n(c),\Sigma_c\cap E)$
and $y=\Pi_{K_n(c)}(x)$.  Since $\partial K_n(c)$ is $C^1$,
$x-y\in N_{K_n(c)}(y)=\R_+\nu(y)$ (Section~\ref{sec:prelim}), so
$x=y+\tau\nu(y)$ with $\tau\ge0$.  If $y_i=0$, then
\[
x_i=y_i+\tau\nu_i(y)=\tau\frac{M'(0)}{\lVert g(y)\rVert}=0.
\]
Thus $\mathcal M_t(K_n(c),\Sigma_c\cap E)\subseteq E$, a Lebesgue null set, for
every $t\ge0$, and \eqref{eq:local-steiner} shows that
$\mathcal C_j(K_n(c),\Sigma_c\cap E)=0$ for all $j$.  Consequently
$V_j(K_n(c))=\mathcal C_j(K_n(c),\Sigma_c^\circ)$, and since $\Sigma_c^\circ$ is
relatively open and $\partial K_n(c)$ is $C^2$ near each of its points,
Lemma~\ref{lem-curvature-density} with $\beta=\Omega=\Sigma_c^\circ$ and
$r=k=n-1-j$ gives
\begin{equation}\label{eq:curvature-integral}
(n-j)\kappa_{n-j}V_j(K_n(c))
=\int_{\Sigma_c^\circ}e_k(\chi)\dd\cH^{n-1},
\qquad k=n-1-j.
\end{equation}
Since all principal curvatures are nonnegative, the integrand is
nonnegative, and the integral is finite because $V_j(K_n(c))$ is.
For $j=n-1$, \eqref{eq:curvature-integral} reads
$2V_{n-1}(K_n(c))=\cH^{n-1}(\Sigma_c^\circ)=\cH^{n-1}(\Sigma_c)$, in
agreement with $\mathcal C_{n-1}(K,\cdot)=\tfrac12\cH^{n-1}\llcorner\partial K$,
see \cite[Section~4.2]{Schneider}.  Here $\Sigma_c\cap E$ is
$\cH^{n-1}$-null, each $\Sigma_c\cap\{x_i=0\}$ being a $C^1$ hypersurface
of the coordinate hyperplane.

\subsection{Coarea and coordinate factorization}\label{sec:coarea}

\begin{proof}[Proof of Proposition~\ref{prop:product}]
Multiply \eqref{eq:curvature-integral} by $e^{-\lambda c}$ and integrate
in $c$.  For $x\notin E$, let $\chi(x)$ denote the curvatures of
$\Sigma_{F_n(x)}$ at $x$.  By Lemma~\ref{lem:curvature},
$e_k(\chi(x))$ is a continuous function of $(M'(x_i))_{i=1}^n$ and
$(M''(x_i))_{i=1}^n$ on the open set $\R^n\setminus E$, where all these
quantities are continuous and $\lVert g\rVert>0$.  In particular, it is
measurable there.  Apply coarea to the nonnegative measurable function
$\psi(x)=e^{-\lambda F_n(x)}e_k(\chi(x))$ off $E$, extended by zero on $E$.
\[
\int_{\R^n}\psi(x)\lVert\nabla F_n(x)\rVert\dd x
=\int_0^\infty\int_{\Sigma_c^\circ}\psi(x)\dd\cH^{n-1}(x)\dd c.
\]
Because $F_n$ is $C^1$, it is Lipschitz on each bounded ball, so the
coarea formula of Section~\ref{sec:prelim} applies on each ball, and
nonnegativity permits passage to $\R^n$ by monotone convergence.  The
factor
$\lVert\nabla F_n\rVert=\lVert g\rVert$ from coarea reduces the
curvature denominator in Lemma~\ref{lem:curvature} by one power.
\[
e_k(\chi)\lVert g\rVert
=\frac{\sum_{i=1}^ng_i^2e_k((h_\ell)_{\ell\ne i})}
{\lVert g\rVert^{k+2}}\lVert g\rVert
=\frac{\sum_{i=1}^ng_i^2e_k((h_\ell)_{\ell\ne i})}
{\lVert g\rVert^{k+1}}.
\]
We therefore obtain
\begin{equation}\label{eq:coarea}
(n-j)\kappa_{n-j}Q_{n,j}(\lambda)
=\int_{\R^n}
\frac{\sum_{i=1}^n g_i^2e_k((h_\ell)_{\ell\ne i})}
{\lVert g\rVert^{k+1}}e^{-\lambda F_n(x)}\dd x.
\end{equation}
The quotient is assigned the value zero on the Lebesgue null set $E$.

It remains to separate the coordinates in $\lVert g\rVert^{-(k+1)}$.
For $a,y>0$, the Gamma identity is
\[
y^{-a}=\frac1{\Gamma(a)}\int_0^\infty\theta^{a-1}e^{-\theta y}\dd\theta.
\]
For $x\notin E$, take $y=\lVert g\rVert^2$ and
$a=(k+1)/2>0$.  This gives
\[
\lVert g\rVert^{-(k+1)}
=\frac1{\Gamma((k+1)/2)}\int_0^\infty
\theta^{(k+1)/2-1}e^{-\theta\lVert g\rVert^2}\dd\theta.
\]
After substitution in \eqref{eq:coarea}, Tonelli permits interchange of the
$x$- and $\theta$-integrals, since $M''\ge0$ and every factor is nonnegative.
Thus the integral to evaluate for each $\theta$ is
\[
\sum_{i=1}^n\int_{\R^n}g_i^2e_k((h_\ell)_{\ell\ne i})
e^{-\lambda F_n(x)-\theta\lVert g(x)\rVert^2}\dd x.
\]
Both $F_n(x)=\sum_qM(x_q)$ and
$\lVert g(x)\rVert^2=\sum_qM'(x_q)^2$ are sums over coordinates.
For fixed $\theta,\lambda>0$, the exponential therefore factors as
\[
e^{-\lambda F_n(x)-\theta\lVert g(x)\rVert^2}
=\prod_{q=1}^n e^{-\lambda M(x_q)-\theta M'(x_q)^2}.
\]
For a fixed $i$, expand
\[
e_k((h_\ell)_{\ell\ne i})
=\sum_{\substack{I\subseteq\{1,\ldots,n\}\setminus\{i\}\\|I|=k}}
\prod_{\ell\in I}M''(x_\ell).
\]
To see the separation for one summand, write
$w(s)=e^{-\lambda M(s)-\theta M'(s)^2}$.  Tonelli gives
\begin{align*}
&\int_{\R^n}M'(x_i)^2
\left(\prod_{\ell\in I}M''(x_\ell)\right)
\prod_{q=1}^nw(x_q)\dd x\\
&\quad=\left(\int_\R M'(s)^2w(s)\dd s\right)
\left(\int_\R M''(s)w(s)\dd s\right)^k
\left(\int_\R w(s)\dd s\right)^{n-1-k}\\
&\quad=AB^kZ^{n-1-k}.
\end{align*}
Coordinate $i$ supplies $A$, the $k$ coordinates in $I$ supply $B^k$,
and the remaining $n-1-k=j$ coordinates supply $Z^j$.
There are $\binom{n-1}{k}$ choices of $I$, so
\begin{equation}\label{eq:coordinate-factorization}
\int_{\R^n}g_i^2e_k((h_\ell)_{\ell\ne i})
e^{-\lambda F_n-\theta\lVert g\rVert^2}\dd x
=\binom{n-1}{k}AB^kZ^{n-1-k}.
\end{equation}
All three integrals are evaluated at $(\theta,\lambda)$.
Tonelli justifies this factorization in the extended nonnegative sense.  Later,
Lemma~\ref{lem:endpoints} shows that $A$, $B$, and $Z$ are finite and
positive.  The right-hand side of \eqref{eq:coordinate-factorization}
does not depend on $i$, so summing over $i=1,\ldots,n$ contributes a
factor $n$.  Combining this with the Gamma representation gives
\begin{equation}\label{eq:product-before-constants}
(n-j)\kappa_{n-j}Q_{n,j}(\lambda)
=\frac{n\binom{n-1}{k}}{\Gamma((k+1)/2)}
\int_0^\infty\theta^{(k+1)/2-1}AB^kZ^{n-1-k}\dd\theta.
\end{equation}
Finally, replace $k$ by $n-1-j$ and use
$\binom{n-1}{k}=\binom{n-1}{j}$.  The definition of $\kappa_{n-j}$ and
$\Gamma(z+1)=z\Gamma(z)$ give
\[
(n-j)\kappa_{n-j}\Gamma((n-j)/2)
=\frac{(n-j)\pi^{(n-j)/2}\Gamma((n-j)/2)}
{\Gamma((n-j)/2+1)}
=2\pi^{(n-j)/2}.
\]
Dividing \eqref{eq:product-before-constants} by $(n-j)\kappa_{n-j}$
therefore gives \eqref{eq:product}, including its prefactor.
\end{proof}

The formula also explains the exponent in \eqref{eq:variational-exponent}.
As $j_n/n\to\alpha$, the powers of $Z$, $B$, and $\theta$ contribute
$\alpha\log Z$, $(1-\alpha)\log B$, and
$(1-\alpha)\log\theta/2$ after division by $n$.
The remaining factor $A/(\theta B)$ has no $n$-dependent exponent, and
the estimates of the next section justify this separation under the
integral.  Stirling's formula gives the prefactor contribution
$H(\alpha)-\frac{1-\alpha}{2}\log\pi$.

\section{Exponential rate of the radius transform}\label{sec:asymptotics}

We now determine the exponential rate of the integral in \eqref{eq:product}.
On compact $\theta$-intervals, the rate is governed by $\phi_\alpha$.
The estimates below control both endpoints and justify taking its supremum.

\subsection{One-dimensional estimates}

\begin{lemma}[Endpoint estimates]\label{lem:endpoints}
\begin{enumerate}[label=\textnormal{(\roman*)},nosep]
\item $Z$ is finite, positive, and jointly continuous on
$[0,\infty)\times(0,\infty)$, and $A$ and $B$ are finite, positive, and
jointly continuous on $(0,\infty)\times(0,\infty)$.  Moreover
$B(\theta,0)\le\sqrt{\pi/\theta}$ for every $\theta>0$.  The first two
$\lambda$-derivatives of $Z$ and $B$ are obtained by differentiation under
the integral sign and are jointly continuous on
$(0,\infty)\times(0,\infty)$.
\item For $\theta>0$ and $\lambda\ge0$,
\begin{equation}\label{eq:Bexact}
\sqrt\theta\,B(\theta,\lambda)
=\int_{-L\sqrt\theta}^{L\sqrt\theta}
e^{-\lambda M(\sigma(u/\sqrt\theta))-u^2}\dd u,
\end{equation}
the interval being $\R$ when $L=\infty$.  In $\theta$, this expression is
nondecreasing, at most $\sqrt\pi$, and tends to $\sqrt\pi$ as
$\theta\to\infty$.  If $\lambda>0$, it tends to $0$ as $\theta\downarrow0$.
If $L=\infty$ and $\lambda=0$, it is identically $\sqrt\pi$.
\item For $\theta,\lambda>0$, $Z(\theta,\lambda)$ is nonincreasing in
$\theta$ and
\begin{equation}\label{eq:thetaends}
A(\theta,\lambda)\le\frac{Z(0,\lambda)}{e\theta},
\qquad
Z(\theta,\lambda)\downarrow0\quad(\theta\to\infty).
\end{equation}
\item Fix $\alpha\in(0,1)$ and a compact interval $J\subset(0,\infty)$.
Then $\phiA(\theta,\lambda)\to-\infty$ as $\theta\downarrow0$ and as
$\theta\to\infty$, uniformly for $\lambda\in J$.  Consequently
$\PhiA(\lambda)$ is finite and attained for every $\lambda>0$, and for
$\lambda\in J$ all maximizers lie in one compact subinterval of
$(0,\infty)$.
\end{enumerate}
\end{lemma}

\begin{proof}
\emph{Absolute continuity of $M'$.}
For $0<t<s$, $M'$ is $C^1$ on $[t,s]$, so
$M'(s)-M'(t)=\int_t^sM''(u)\dd u$.  Letting $t\downarrow0$, continuity of
$M'$ and monotone convergence ($M''\ge0$) give
\[
M'(s)=M'(0)+\int_0^sM''(u)\dd u\qquad(s\ge0),
\]
and similarly for $s\le0$.  Hence $M''$ is locally integrable, even at the
origin, and $M'$ is locally absolutely continuous on $\R$.  Moreover
$M'(0)=0$, because $M'$ is odd.

\emph{The substitution $r=M'(s)$.}
Since $M'$ is nondecreasing and locally absolutely continuous, the
one-dimensional area formula \cite[Section~3.3]{EvansGariepy} shows that
the image of the measure
$M''(s)\dd s$ under $M'$ is Lebesgue measure on $(-L,L)$.  For all but at
most countably many $r$, the level set $\{s:M'(s)=r\}$ is the singleton
$\{\sigma(r)\}$.  On every non-singleton level interval, $M''$ vanishes
in the interior, while the endpoints of these intervals form a countable
null set.  Hence
$\int_\R\psi(M(s),M'(s))M''(s)\dd s=\int_{-L}^L\psi(M(\sigma(r)),r)\dd r$
for every nonnegative measurable $\psi$.  In particular,
\begin{equation}\label{eq:B-subst}
B(\theta,\lambda)=\int_{-L}^L e^{-\lambda M(\sigma(r))-\theta r^2}\dd r
\qquad(\theta,\lambda\ge0).
\end{equation}

\emph{(i) Integrability and parameter regularity.}
Fix $J=[\lambda_-,\lambda_+]\subset(0,\infty)$ and $\theta_0>0$.  The
linear lower growth of $M$ gives $\int_\R e^{-\eta M}<\infty$ for every
$\eta>0$, and \eqref{eq:B-subst} with $\lambda=0$ gives
$B(\theta_0,0)=\int_\R M''e^{-\theta_0M'^2}\le\sqrt{\pi/\theta_0}$, which
is the stated bound on $B(\theta,0)$.  Since
$\sup_{u\ge0}u^qe^{-\lambda_-u/2}<\infty$, $e^{-\lambda_-M/2}\le1$, and
$xe^{-\theta_0x}\le1/(e\theta_0)$ for $x\ge0$, we have, for $q=0,1,2$,
$\lambda\in J$, and $\theta\ge\theta_0$,
\begin{gather*}
M^qe^{-\lambda M}\le C_qe^{-\lambda_-M/2},\qquad
M^qM''e^{-\lambda M-\theta M'^2}\le C_qM''e^{-\theta_0M'^2},\\
M^qM'^2e^{-\lambda M-\theta M'^2}\le\frac{C_q}{e\theta_0}e^{-\lambda_-M/2}.
\end{gather*}
Since $e^{-\theta M'^2}\le1$, the first bound also dominates the integrands
of $Z$ and of its first two $\lambda$-derivatives for all $\theta\ge0$.
The other two bounds dominate, for $\theta\ge\theta_0$, the integrands of
$B$, $A$, and the first two $\lambda$-derivatives of $B$, which carry the
additional factors $-M$ and $M^2$.  Dominated convergence therefore proves
the finiteness, joint continuity, and differentiability claims.
Positivity of $Z$ and $A$ follows from $M'(s)\ne0$ for $s\ne0$, and
positivity of $B$ from \eqref{eq:B-subst}.

\emph{(ii) The function $\sqrt\theta B$.}
For $\theta>0$, the substitution $r=u/\sqrt\theta$ in \eqref{eq:B-subst}
gives \eqref{eq:Bexact}.  Extend its integrand by zero outside
$(-L\sqrt\theta,L\sqrt\theta)$.  The function
$r\mapsto M(\sigma(r))$ is even, nondecreasing in $|r|$, and tends to
$0$ as $r\to0$, because $\sigma$ is monotone and $M'(s)\ne0$ for $s\ne0$
force $\sigma(r)\to0$.  The extended integrand is therefore
nondecreasing in $\theta$, bounded by $e^{-u^2}$, and converges to
$e^{-u^2}$ as $\theta\to\infty$.  Monotone convergence proves the stated
properties at infinity.  When $L=\infty$ and $\lambda=0$, the integrand is
already $e^{-u^2}$.  If $\lambda>0$ and $u\ne0$, the extended integrand
tends to $0$ as $\theta\downarrow0$.  Indeed, it vanishes once $|u|\ge L\sqrt\theta$
when $L<\infty$, and $M(\sigma(u/\sqrt\theta))\to\infty$ when $L=\infty$.
Dominated convergence gives $\sqrt\theta B(\theta,\lambda)\to0$.

\emph{(iii) The bounds for $A$ and $Z$.}
Applying $xe^{-\theta x}\le1/(e\theta)$, valid for $x\ge0$, with
$x=M'(s)^2$ gives
\[
A(\theta,\lambda)
\le\frac1{e\theta}\int_\R e^{-\lambda M(s)}\dd s
=\frac{Z(0,\lambda)}{e\theta}.
\]
The integrand of $Z$ is dominated by $e^{-\lambda M}$, is nonincreasing
in $\theta$, and tends to $0$ for every $s\ne0$, because $M'(s)\ne0$
there.  Hence $Z$ is nonincreasing in $\theta$, and dominated convergence
gives $Z(\theta,\lambda)\downarrow0$ as $\theta\to\infty$.

\emph{(iv) Endpoint behavior and attainment.}
For $\theta>0$ we may rewrite $\phiA$ and use that $Z$ and $B$ are
nonincreasing in $\lambda$, which gives
\[
\phiA(\theta,\lambda)
=\alpha\log Z(\theta,\lambda)+(1-\alpha)\log\bigl(\sqrt\theta\,B(\theta,\lambda)\bigr)
\le\alpha\log Z(\theta,\lambda_-)
+(1-\alpha)\log\bigl(\sqrt\theta\,B(\theta,\lambda_-)\bigr).
\]
As $\theta\downarrow0$, the first term is at most $\alpha\log Z(0,\lambda_-)$
and the second tends to $-\infty$, since $\alpha<1$ and
$\sqrt\theta B(\theta,\lambda_-)\to0$.  As $\theta\to\infty$, the second
term is at most $\frac{1-\alpha}{2}\log\pi$ and the first tends to
$-\infty$, since $\alpha>0$ and $Z(\theta,\lambda_-)\to0$.  Hence
$\phiA(\theta,\lambda)\to-\infty$ at both endpoints, uniformly for
$\lambda\in J$.  Put
$\tau_J=\min_{\lambda\in J}\phiA(1,\lambda)$, which is finite by
continuity.  The uniform endpoint limits allow $0<\theta_-<1<\theta_+$
such that $\phiA(\theta,\lambda)<\tau_J-1$ whenever
$\theta\notin[\theta_-,\theta_+]$ and $\lambda\in J$.  These values are
strictly below the available value $\phiA(1,\lambda)\ge \tau_J$.  For each
$\lambda\in J$, continuity now gives a maximum on $[\theta_-,\theta_+]$,
which is also the global maximum.
The same strict comparison shows that no maximizer lies outside this
common compact interval.
\end{proof}

The maximizer in Lemma~\ref{lem:endpoints} is in fact unique.  The
argument is that of \cite[Proposition~2.12]{KabluchkoMarynychProchno},
and rests on an integration-by-parts identity between three of the
one-dimensional integrals.

\begin{lemma}[Inner maximizer]\label{lem:inner}
Fix $\alpha\in(0,1)$.  For $\theta,\lambda>0$ put
\[
D(\theta,\lambda)=\int_{\R}M'(s)^2M''(s)
e^{-\lambda M(s)-\theta M'(s)^2}\dd s.
\]
Then $D$ is finite, $Z_\theta=-A$, $B_\theta=-D$, and
\begin{equation}\label{eq:ABC}
B=\lambda A+2\theta D.
\end{equation}
For every $\lambda>0$, the function $\theta\mapsto\phiA(\theta,\lambda)$
has a unique maximizer $\theta_*(\lambda)\in(0,\infty)$.  It is the unique
solution of
\begin{equation}\label{eq:inner-stationarity}
2\alpha\theta B(\theta,\lambda)=(1-\alpha)\lambda Z(\theta,\lambda),
\end{equation}
and $\lambda\mapsto\theta_*(\lambda)$ is continuous.  In particular,
$\PhiA(\lambda)=\phiA(\theta_*(\lambda),\lambda)$.
\end{lemma}

\begin{proof}
\emph{The identity \eqref{eq:ABC}.}
The substitution $r=M'(s)$ from Lemma~\ref{lem:endpoints} gives
\[
D(\theta,\lambda)
=\int_{-L}^{L}r^2e^{-\lambda M(\sigma(r))-\theta r^2}\dd r
\le\int_{\R}r^2e^{-\theta r^2}\dd r<\infty.
\]
The product rule gives, away from zero,
\begin{align*}
\frac{\dd}{\dd s}\left(M'e^{-\lambda M-\theta M'^2}\right)
&=M''e^{-\lambda M-\theta M'^2}
-M'(\lambda M'+2\theta M'M'')e^{-\lambda M-\theta M'^2}\\*
&=(M''-\lambda M'^2-2\theta M'^2M'')
e^{-\lambda M-\theta M'^2}.
\end{align*}
The three terms on the right integrate to $B$, $-\lambda A$, and
$-2\theta D$, respectively.  Apply the fundamental theorem on
$[-R,-\varepsilon]\cup[\varepsilon,R]$.  The boundary terms at $\pm R$
vanish as $R\to\infty$.  If $L<\infty$, this follows from boundedness of
$M'$ and $M(s)\to\infty$.  If $L=\infty$, it follows from
$|M'|e^{-\lambda M-\theta M'^2}\le |M'|e^{-\theta M'^2}\to0$.
The boundary terms at $\pm\varepsilon$ vanish as $\varepsilon\downarrow0$ because
$M'(\pm\varepsilon)\to0$, and the integrability of $A,B,D$ justifies both
limits.  Since the integral of the derivative is zero, \eqref{eq:ABC}
follows.

\emph{Derivatives in $\theta$.}
Fix $\theta_0>0$.  On $\theta\in[\theta_0/2,3\theta_0/2]$, the absolute
derivatives of the integrands of $Z$ and $B$ are dominated by
$2e^{-\lambda M(s)}/(e\theta_0)$ and, after the substitution $r=M'(s)$,
by $r^2e^{-\theta_0r^2/2}$.  Both bounds are integrable, so
differentiation under the integral sign gives $Z_\theta=-A$ and
$B_\theta=-D$.  Substituting these into the derivative of
\eqref{eq:variational-exponent} and using
$D=(B-\lambda A)/(2\theta)$ from \eqref{eq:ABC},
\begin{align*}
\partial_\theta\phiA
&=\frac{1-\alpha}{2\theta}
+\alpha\frac{Z_\theta}{Z}+(1-\alpha)\frac{B_\theta}{B}
=\frac{1-\alpha}{2\theta}-\alpha\frac AZ
-\frac{1-\alpha}{2\theta B}(B-\lambda A)\\
&=A\left(\frac{(1-\alpha)\lambda}{2\theta B}-\frac\alpha Z\right)
=\frac{\alpha A}{Z}\left(\frac{\xi_\lambda}{\gamma_\lambda(\theta)}-1\right),
\end{align*}
where
\[
\xi_\lambda=\frac{(1-\alpha)\lambda}{2\alpha},
\qquad
\gamma_\lambda(\theta)=\frac{\theta B(\theta,\lambda)}{Z(\theta,\lambda)}.
\]
Since $A,Z>0$, the sign of $\partial_\theta\phiA(\theta,\lambda)$ is the
sign of $\xi_\lambda-\gamma_\lambda(\theta)$.

\emph{Monotonicity of $\gamma_\lambda$.}
By $Z_\theta=-A$, $B_\theta=-D$, and \eqref{eq:ABC},
\[
\gamma_\lambda'(\theta)
=\frac{B-\theta D}{Z}+\frac{\theta AB}{Z^2}
=\frac{B+\lambda A}{2Z}+\frac{\theta AB}{Z^2}>0.
\]
By Lemma~\ref{lem:endpoints}, $\gamma_\lambda$ is continuous on $(0,\infty)$,
and in the factorization
\[
\gamma_\lambda(\theta)
=\sqrt\theta\,\frac{\sqrt\theta\,B(\theta,\lambda)}{Z(\theta,\lambda)}
\]
we have $\sqrt\theta B\le\sqrt\pi$ and $Z(\theta,\lambda)\to Z(0,\lambda)>0$
as $\theta\downarrow0$, while $\sqrt\theta B\to\sqrt\pi$ and $Z\to0$ as
$\theta\to\infty$.  Hence $\gamma_\lambda(\theta)\to0$ as $\theta\downarrow0$
and $\gamma_\lambda(\theta)\to\infty$ as $\theta\to\infty$.  Thus $\gamma_\lambda$ is an
increasing bijection of
$(0,\infty)$ onto itself, and $\gamma_\lambda(\theta)=\xi_\lambda$, which is
\eqref{eq:inner-stationarity}, has exactly one solution
$\theta_*(\lambda)$.  The function $\phiA(\cdot,\lambda)$ is strictly
increasing on $(0,\theta_*(\lambda))$ and strictly decreasing on
$(\theta_*(\lambda),\infty)$, so $\theta_*(\lambda)$ is its unique
maximizer.

\emph{Continuity.}
Let $\lambda_i\to\lambda$.  By Lemma~\ref{lem:endpoints} the
maximizers $\theta_*(\lambda_i)$ lie in a compact subinterval of
$(0,\infty)$, and joint continuity of $\phiA$ shows that every limit
point of $(\theta_*(\lambda_i))$ maximizes $\phiA(\cdot,\lambda)$, hence
equals $\theta_*(\lambda)$.  Thus $\theta_*$ is continuous.
\end{proof}

\subsection{A direct Laplace principle}

Set $\alpha_n=j_n/n\in(0,1)$ for all sufficiently large $n$, and write
the integral in \eqref{eq:product} as
\begin{equation}\label{eq:In-exponent-form}
\begin{aligned}
I_n(\lambda)
&=\int_0^\infty\theta^{(n-j_n)/2-1}Z^{j_n}B^{n-1-j_n}A\dd\theta\\
&=\int_0^\infty\frac{A(\theta,\lambda)}{\theta B(\theta,\lambda)}
e^{n\phi_{\alpha_n}(\theta,\lambda)}\dd\theta.
\end{aligned}
\end{equation}
Indeed, $e^{n\phi_{\alpha_n}}=
\theta^{(n-j_n)/2}Z^{j_n}B^{n-j_n}$, so division by $\theta B$
recovers the exact powers in the first integral.  This leaves a factor
$A/(\theta B)$ that does not depend on $n$.

\begin{proposition}[Radius-transform rate]\label{prop:transform-rate}
Let $j_n\in\{0,\ldots,n-1\}$ satisfy
$j_n/n\to\alpha\in(0,1)$, and let $I_n(\lambda)$ be defined by
\eqref{eq:In-exponent-form}.  Then, for every fixed $\lambda>0$,
\begin{equation}\label{eq:Inrate}
\lim_{n\to\infty}\frac1n\log I_n(\lambda)=\PhiA(\lambda).
\end{equation}
Consequently, with $\Lambda_n(\lambda)=n^{-1}\log Q_{n,j_n}(\lambda)$,
\begin{equation}\label{eq:transform-rate}
\Lambda_n(\lambda)\longrightarrow
\Lambda(\lambda)=H(\alpha)-\frac{1-\alpha}{2}\log\pi+\PhiA(\lambda).
\end{equation}
\end{proposition}

\begin{proof}
Fix $\lambda>0$.  The difference of the exponents has the exact form
\[
\begin{aligned}
&\phi_{\alpha_n}(\theta,\lambda)-\phiA(\theta,\lambda)\\
&\quad=(\alpha_n-\alpha)
\left(-\frac12\log\theta+\log Z(\theta,\lambda)-\log B(\theta,\lambda)\right).
\end{aligned}
\]
On each compact interval $U\subset(0,\infty)$, the expression in
parentheses is continuous and bounded.  Hence
$\sup_U|\phi_{\alpha_n}-\phiA|\le C_U|\alpha_n-\alpha|\to0$.
The factor
$\Delta(\theta)=A(\theta,\lambda)/(\theta B(\theta,\lambda))$ is continuous
and positive.

\emph{Lower bound.}
Given $\varepsilon>0$, let $\theta_*$ be a maximizer of
$\phiA(\cdot,\lambda)$ (Lemma~\ref{lem:endpoints}) and choose a
closed interval $U\subset(0,\infty)$ of positive length around it such
that $\phiA\ge\PhiA(\lambda)-\varepsilon$ on $U$.  Continuity and
positivity of $\Delta$ give $d_U:=\min_U \Delta>0$.
Uniform convergence gives $\phi_{\alpha_n}\ge
\PhiA(\lambda)-2\varepsilon$ on $U$ for all large $n$, and therefore
\[
I_n(\lambda)\ge |U|d_U\,
e^{n(\PhiA(\lambda)-2\varepsilon)}.
\]
Taking logarithms yields
\[
\frac1n\log I_n(\lambda)
\ge\PhiA(\lambda)-2\varepsilon+\frac{\log(|U|d_U)}n.
\]
The final term tends to zero because $U$ and $d_U$ are fixed.
Thus $\liminf_n n^{-1}\log I_n(\lambda)\ge
\PhiA(\lambda)-2\varepsilon$.  Letting $\varepsilon\downarrow0$ gives
the lower bound.

\emph{Upper bound near zero.}
We first bound $\int_0^\infty AB\dd\theta$.  By Tonelli and
$\int_0^\infty e^{-\theta y}\dd\theta=1/y$,
\[
\int_0^\infty A(\theta,\lambda)B(\theta,\lambda)\dd\theta
=\int_\R\int_\R\frac{M'(s)^2M''(t)}{M'(s)^2+M'(t)^2}
e^{-\lambda M(s)-\lambda M(t)}\dd t\dd s.
\]
For fixed $s\ne0$, the substitution $r=M'(t)$ of
Lemma~\ref{lem:endpoints} and $e^{-\lambda M(t)}\le1$ bound the inner
integral by
$M'(s)^2e^{-\lambda M(s)}\int_\R(M'(s)^2+r^2)^{-1}\dd r
=\pi|M'(s)|e^{-\lambda M(s)}$.  Since
$\int_\R|M'|e^{-\lambda M}\dd s=2/\lambda$, we obtain
\begin{equation}\label{eq:AB-integral}
\int_0^\infty A(\theta,\lambda)B(\theta,\lambda)\dd\theta
\le\frac{2\pi}\lambda.
\end{equation}
Now put $k_n=n-1-j_n$, which is at least $1$ for all large $n$.  The
integrand in the first line of \eqref{eq:In-exponent-form} equals
\[
\theta^{(k_n-1)/2}Z^{j_n}B^{k_n}A
=Z^{j_n}\bigl(\sqrt\theta B\bigr)^{k_n-1}AB.
\]
By Lemma~\ref{lem:endpoints}, $Z$ is nonincreasing in $\theta$ and
$\sqrt\theta B$ is nondecreasing.  For $0<a\le1$, put $Z_0=Z(0,\lambda)$ and
$\omega(a)=\sqrt a\,B(a,\lambda)$.  Then \eqref{eq:AB-integral} gives
\[
I_n^{-}(a):=\int_0^a\theta^{(n-j_n)/2-1}Z^{j_n}B^{n-1-j_n}A\dd\theta
\le\frac{2\pi}\lambda Z_0^{j_n}\omega(a)^{k_n-1}.
\]
Since $j_n/n\to\alpha$ and $(k_n-1)/n\to1-\alpha$, it follows that
\begin{equation}\label{eq:left-tail-rate}
\limsup_{n\to\infty}\frac1n\log I_n^{-}(a)
\le\alpha\log Z_0+(1-\alpha)\log\omega(a).
\end{equation}
By Lemma~\ref{lem:endpoints}, $\omega(a)\to0$ as $a\downarrow0$, so the
right-hand side tends to $-\infty$.

\emph{Upper bound at infinity.}
Let $b\ge1$.  For $\theta\ge b$, monotonicity in $\theta$ gives
$Z(\theta,\lambda)\le Z(b,\lambda)$, and Lemma~\ref{lem:endpoints} gives
$B\le\sqrt{\pi/\theta}$ and $A\le C_A\theta^{-1}$ with $C_A=Z(0,\lambda)/e$.
The powers of $\theta$ add up to
$\frac{n-j_n}{2}-1-\frac{n-1-j_n}{2}-1=-\frac32$, so
\[
\theta^{(n-j_n)/2-1}Z^{j_n}B^{n-1-j_n}A
\le C_A\pi^{(n-1-j_n)/2}Z(b,\lambda)^{j_n}\theta^{-3/2}
\qquad(\theta\ge b).
\]
Integrating,
\begin{align*}
I_n^{+}(b)
&:=\int_b^\infty\theta^{(n-j_n)/2-1}Z^{j_n}B^{n-1-j_n}A\dd\theta\\
&\le 2b^{-1/2}C_A\,\pi^{(n-1-j_n)/2}Z(b,\lambda)^{j_n}.
\end{align*}
After division by $n$, the logarithm of the fixed factor $2b^{-1/2}C_A$
tends to zero, and therefore
\begin{equation}\label{eq:right-tail-rate}
\limsup_{n\to\infty}\frac1n\log I_n^{+}(b)
\le\frac{1-\alpha}{2}\log\pi+\alpha\log Z(b,\lambda).
\end{equation}
Since $\alpha>0$ and $Z(b,\lambda)\to0$ by \eqref{eq:thetaends}, the
right-hand side tends to $-\infty$ as $b\to\infty$.

\emph{The middle interval and the sum.}
Write $I_n^{0}(a,b)$ for the integral over $[a,b]$.  It is at most
\[
(b-a)\max_{[a,b]}\Delta\,
\exp\left\{n\sup_{[a,b]}\phi_{\alpha_n}(\theta,\lambda)\right\}.
\]
The logarithm of the fixed factor $(b-a)\max_{[a,b]}\Delta$, divided by $n$,
tends to zero.  Uniform convergence of the exponents therefore gives
\[
\limsup_{n\to\infty}\frac1n\log I_n^{0}(a,b)
\le\sup_{[a,b]}\phiA(\theta,\lambda)\le\PhiA(\lambda).
\]
Choose $a$ small and $b$ large
so that \eqref{eq:left-tail-rate} and \eqref{eq:right-tail-rate} are
strictly below $\PhiA(\lambda)$.  These choices are fixed as $n\to\infty$.
Since $I_n=I_n^{-}(a)+I_n^{0}(a,b)+I_n^{+}(b)$,
\[
\frac1n\log I_n
\le\frac{\log3}{n}
+\max\left\{\frac1n\log I_n^{-}(a),\frac1n\log I_n^{0}(a,b),
\frac1n\log I_n^{+}(b)\right\}.
\]
Taking the upper limit and using the three bounds proves
$\limsup_n n^{-1}\log I_n(\lambda)\le\PhiA(\lambda)$.
This proves \eqref{eq:Inrate}.  The argument does not use the uniqueness
of the inner maximizer established in Lemma~\ref{lem:inner}.

\emph{The prefactor.}
The normalized logarithm of the prefactor in \eqref{eq:product} is
\[
\begin{aligned}
\frac1n\log\frac{n\binom{n-1}{j_n}}{2\pi^{(n-j_n)/2}}
&=\frac{\log n-\log2}{n}
+\frac1n\log\binom{n-1}{j_n}
\\&\quad-\frac{1-\alpha_n}{2}\log\pi.
\end{aligned}
\]
Stirling's formula and $j_n/(n-1)\to\alpha\in(0,1)$ give
$n^{-1}\log\binom{n-1}{j_n}\to H(\alpha)$.  The first term tends to
zero, and the final term converges directly.  Hence
\[
\frac1n\log\frac{n\binom{n-1}{j_n}}{2\pi^{(n-j_n)/2}}
\longrightarrow H(\alpha)-\frac{1-\alpha}{2}\log\pi.
\]
Together with Proposition~\ref{prop:product} and \eqref{eq:Inrate}, this
proves \eqref{eq:transform-rate}.
\end{proof}

\subsection{The outer minimization}

For $\lambda>0$ and $\theta>0$, and for $\theta=0$ in the case of
$\mu_Z$, define the probability measures
\[
\mu_{Z,\theta,\lambda}(\dd s)
=\frac{e^{-\lambda M(s)-\theta M'(s)^2}}{Z(\theta,\lambda)}\dd s,
\qquad
\mu_{B,\theta,\lambda}(\dd s)
=\frac{M''(s)e^{-\lambda M(s)-\theta M'(s)^2}}
{B(\theta,\lambda)}\dd s.
\]
Write $\E_Z,\E_B$ and $\Var_Z,\Var_B$ for their expectations and
variances when the parameters are clear.

\begin{lemma}[Convexity and coercivity]\label{lem:coercive}
Fix $\alpha\in(0,1)$.  The function $\Lambda$ is finite, nonincreasing,
and strictly convex on
$(0,\infty)$.  For every $\rho>0$,
\[
\Lambda(\lambda)\to+\infty\quad(\lambda\downarrow0),\qquad
\lambda\rho+\Lambda(\lambda)\to+\infty\quad(\lambda\to\infty).
\]
Consequently $\lambda\rho+\Lambda(\lambda)$ is coercive on $(0,\infty)$,
meaning it tends to $+\infty$ at both endpoints, and has a unique
minimizer there.
\end{lemma}

\begin{proof}
\emph{Convexity and monotonicity.}
Finiteness of $\Lambda$ is Lemma~\ref{lem:endpoints}(iv), and
Lemma~\ref{lem:endpoints}(i) justifies differentiation under the integrals.
The calculation for $Z$ is
\begin{align*}
Z_\lambda
&=-\int_\R M(s)e^{-\lambda M(s)-\theta M'(s)^2}\dd s
=-Z\E_Z[M],\\
Z_{\lambda\lambda}
&=\int_\R M(s)^2e^{-\lambda M(s)-\theta M'(s)^2}\dd s
=Z\E_Z[M^2].
\end{align*}
It follows that
\[
\partial_\lambda^2\log Z
=\frac{Z_{\lambda\lambda}}Z-
\left(\frac{Z_\lambda}Z\right)^2
=\E_Z[M^2]-\E_Z[M]^2=\Var_Z(M).
\]
The calculation for $B$ inserts the nonnegative weight $M''(s)$ into
both integrals and normalizes by $B$.  Thus
\begin{align*}
\partial_\lambda\log Z&=-\E_Z[M],
&\partial_\lambda^2\log Z&=\Var_Z(M),\\
\partial_\lambda\log B&=-\E_B[M],
&\partial_\lambda^2\log B&=\Var_B(M).
\end{align*}
It follows that
\begin{align*}
\partial_\lambda\phiA
&=-\alpha\E_Z[M]-(1-\alpha)\E_B[M]\le0,\\
\partial_\lambda^2\phiA
&=\alpha\Var_Z(M)+(1-\alpha)\Var_B(M)>0.
\end{align*}
Strictness holds because $\alpha>0$, the $Z$-density is positive on $\R$,
and $M$ is not constant.  Thus each $\phiA(\theta,\cdot)$ is
nonincreasing and strictly convex.

The supremum preserves monotonicity.  To prove strict convexity, let
$\lambda_t=(1-t)\lambda_0+t\lambda_1$ with $0<t<1$ and
$\lambda_0\ne\lambda_1$, and let $\theta_t=\theta_*(\lambda_t)$.  Then
\begin{align*}
\PhiA(\lambda_t)
&=\phiA(\theta_t,\lambda_t)\\
&<(1-t)\phiA(\theta_t,\lambda_0)+t\phiA(\theta_t,\lambda_1)\\
&\le(1-t)\PhiA(\lambda_0)+t\PhiA(\lambda_1).
\end{align*}
The same properties hold for $\Lambda$, which differs from $\PhiA$ by a
constant.

\emph{The endpoint $\lambda\downarrow0$.}
Fix $\theta>0$.  As $\lambda\downarrow0$, the integrands defining
$Z(\theta,\lambda)$ and $B(\theta,\lambda)$ increase, so monotone
convergence gives $Z(\theta,\lambda)\uparrow Z(\theta,0)\in(0,\infty]$
and $B(\theta,\lambda)\uparrow B(\theta,0)>0$.

If $L<\infty$, then $M'(s)^2\to L^2$ as $|s|\to\infty$, so
$Z(\theta,0)=\infty$.  Thus
$\phiA(\theta,\lambda)\to\infty$ for every fixed $\theta>0$, and hence
$\PhiA(\lambda)\to\infty$.

Suppose now that $L=\infty$.  Equation~\eqref{eq:Bexact} gives
$B(\theta,0)=\sqrt{\pi/\theta}$, and therefore
\begin{align*}
\lim_{\lambda\downarrow0}\phiA(\theta,\lambda)
&=\frac{1-\alpha}{2}\log\theta+\alpha\log Z(\theta,0)
+(1-\alpha)\log\sqrt{\frac\pi\theta}\\
&=\alpha\log Z(\theta,0)+\frac{1-\alpha}{2}\log\pi
=:\phiA(\theta,0)\in(-\infty,+\infty].
\end{align*}
The two $\log\theta$ terms cancel, and $\phiA(\theta,0)=+\infty$ exactly
when $Z(\theta,0)=\infty$.  This case does occur.  For the admissible
defining function
\[
M(s)=\int_0^{|s|}\sqrt{\log(1+t)}\dd t,
\]
we have $M'(s)=\operatorname{sgn}(s)\sqrt{\log(1+|s|)}$ and hence
$Z(\theta,0)=2\int_0^\infty(1+s)^{-\theta}\dd s=\infty$ for
$0<\theta\le1$.  Another application of monotone convergence
gives
\[
Z(\theta,0)=\int_\R e^{-\theta M'(s)^2}\dd s
\longrightarrow\int_\R1\dd s=\infty\qquad(\theta\downarrow0),
\]
so $\phiA(\theta,0)\to+\infty$ as $\theta\downarrow0$.  Given $R>0$,
first choose $\theta_R>0$ with $\phiA(\theta_R,0)>R+1$ (possibly
$\phiA(\theta_R,0)=+\infty$).  Since
$\phiA(\theta_R,\lambda)\to\phiA(\theta_R,0)$ as $\lambda\downarrow0$, we
have $\PhiA(\lambda)\ge\phiA(\theta_R,\lambda)>R$ for all sufficiently
small $\lambda$.  Hence $\PhiA(\lambda)\to\infty$, and the same is true of
$\Lambda$, which differs from $\PhiA$ by a constant.

\emph{The endpoint $\lambda\to\infty$.}
Fix $0<\varepsilon<L$.  By evenness and
monotonicity, $M(s)\le M(\varepsilon)$ and $M'(s)^2\le M'(\varepsilon)^2$
for $|s|\le\varepsilon$, and $M(\sigma(r))\le M(\sigma(\varepsilon))$ for
$|r|\le\varepsilon$.  Restricting the integral defining $Z(1,\lambda)$ to
$|s|\le\varepsilon$, and the integral \eqref{eq:B-subst} for $B(1,\lambda)$
to $|r|\le\varepsilon$, gives
\[
Z(1,\lambda)\ge2\varepsilon\,e^{-\lambda M(\varepsilon)-M'(\varepsilon)^2},
\qquad
B(1,\lambda)\ge2\varepsilon\,e^{-\lambda M(\sigma(\varepsilon))-\varepsilon^2}.
\]
Evaluating the supremum defining $\PhiA(\lambda)$ at $\theta=1$, where the
$\log\theta$ term vanishes, therefore yields
\[
\lambda\rho+\Lambda(\lambda)
\ge\lambda\rho+H(\alpha)-\frac{1-\alpha}{2}\log\pi+\phiA(1,\lambda)
\ge\lambda(\rho-\eta_\varepsilon)-C_\varepsilon,
\]
where $\eta_\varepsilon=\alpha M(\varepsilon)+(1-\alpha)M(\sigma(\varepsilon))$
and $C_\varepsilon$ does not depend on $\lambda$.  Since $M$ is continuous
with $M(0)=0$ and $\sigma(r)\to0$ as $r\to0$, we have $\eta_\varepsilon\to0$
as $\varepsilon\downarrow0$.  Choosing $\varepsilon$ with
$\eta_\varepsilon\le\rho/2$ gives
$\lambda\rho+\Lambda(\lambda)\ge\lambda\rho/2-C_\varepsilon\to\infty$.
Finally, a finite convex function is continuous on its open domain.
The two endpoint limits therefore ensure an attained interior minimum
of $\lambda\rho+\Lambda(\lambda)$, and strict convexity makes it unique.
\end{proof}

\begin{lemma}[Optimality equations]\label{lem:optimality}
Fix $\alpha\in(0,1)$.  The function $\PhiA$ is differentiable on
$(0,\infty)$, with
\begin{equation}\label{eq:Phi-derivative}
\PhiA'(\lambda)=-\alpha\E_Z[M]-(1-\alpha)\E_B[M]
\qquad\text{at }(\theta,\lambda)=(\theta_*(\lambda),\lambda).
\end{equation}
For every $\rho>0$, let $\lambda_\rho$ be the minimizer of
$\lambda\rho+\PhiA(\lambda)$ from Lemma~\textup{\ref{lem:coercive}}.  The pair
$(\theta,\lambda)=(\theta_*(\lambda_\rho),\lambda_\rho)$ is the unique
solution in $(0,\infty)^2$ of
\begin{equation}\label{eq:optimality}
2\alpha\theta B(\theta,\lambda)=(1-\alpha)\lambda Z(\theta,\lambda),
\qquad
\alpha\E_Z[M]+(1-\alpha)\E_B[M]=\rho,
\end{equation}
where $\E_Z$ and $\E_B$ refer to $\mu_{Z,\theta,\lambda}$ and
$\mu_{B,\theta,\lambda}$.  In the notation of Theorem~\textup{\ref{thm:main}},
\eqref{eq:optimality} is \eqref{eq-optimality-main}.
\end{lemma}

\begin{proof}
Fix $\lambda>0$ and put $\theta_\delta=\theta_*(\lambda+\delta)$ for small
$\delta\ne0$.  Comparing the suprema defining $\PhiA(\lambda+\delta)$ and
$\PhiA(\lambda)$ with their values at $\theta_*(\lambda)$ and at
$\theta_\delta$ gives
\[
\phiA(\theta_*(\lambda),\lambda+\delta)-\phiA(\theta_*(\lambda),\lambda)
\le\PhiA(\lambda+\delta)-\PhiA(\lambda)
\le\phiA(\theta_\delta,\lambda+\delta)-\phiA(\theta_\delta,\lambda).
\]
For $\delta>0$, divide by $\delta$.  For $\delta<0$, the two inequalities
reverse after division.  The mean value theorem applies to each bounding
difference.  Since $\partial_\lambda\phiA$ is jointly continuous
(Lemma~\ref{lem:endpoints}) and $\theta_*$ is continuous
(Lemma~\ref{lem:inner}), both bounds converge to
$\partial_\lambda\phiA(\theta_*(\lambda),\lambda)$.  Hence $\PhiA$ is
differentiable, and the formula for $\partial_\lambda\phiA$ in the proof
of Lemma~\ref{lem:coercive} gives \eqref{eq:Phi-derivative}.

By Lemma~\ref{lem:coercive}, $\lambda\rho+\PhiA(\lambda)$ is strictly
convex with a unique interior minimizer $\lambda_\rho$, which is therefore
the unique solution of $\PhiA'(\lambda)=-\rho$.  Given $\lambda>0$, the
first equation in \eqref{eq:optimality} is \eqref{eq:inner-stationarity}
and holds exactly for $\theta=\theta_*(\lambda)$.  At this $\theta$, the
second equation reads $\PhiA'(\lambda)=-\rho$ by
\eqref{eq:Phi-derivative}, so it holds exactly for $\lambda=\lambda_\rho$.
\end{proof}

\begin{remark}[Curvature measures]\label{rem:optimality}
The second equation in \eqref{eq:optimality} says that the mixture
$\alpha\mu_{Z,\theta,\lambda}+(1-\alpha)\mu_{B,\theta,\lambda}$ satisfies,
on average over a coordinate, the constraint $\sum_iM(x_i)=n\rho$ that
defines $\partial K_n(n\rho)$.  For $M(s)=|s|^p$ and after the scaling of
Section~\ref{sec:known-cases}, this mixture is exactly the limit law of a
single coordinate of a random point of $\partial K_n(n\rho)$ distributed
according to the normalized curvature measure
$\mathcal C_{j_n}(K_n(n\rho),\cdot)/V_{j_n}(K_n(n\rho))$, see
\cite[Theorem~3.1(i)]{KabluchkoMarynychProchno}.  This suggests an
analogous Maxwell--Poincar\'e--Borel law for Orlicz balls, see
Section~\ref{sec-open}.
\end{remark}
\section{Exact-mean inversion}\label{sec:inversion}

We use two standard facts about one-dimensional convexity.  A real-valued
convex function on an open interval is continuous
\cite[Theorem~10.1]{Rockafellar}, and multiplication by the exponential
of an affine function preserves log-concavity.  A log-concave density is
understood in its upper-semicontinuous version, obtained by taking upper
limits at the endpoints of its positive support.  Thus its point values
are well defined and its essential supremum is attained.

The transform integrates the profile against $e^{-\lambda c}$ over all
levels $c$.  Writing $f_n(c)=V_{j_n}(K_n(c))$ and $Q_n=Q_{n,j_n}$ for
its transform, we recover $f_n(n\rho)$ by an exponential change of
measure.  For $\lambda>0$, the tilted probability density
\[
p_{n,\lambda}(c)=\frac{e^{-\lambda c}f_n(c)}{Q_n(\lambda)},
\qquad c>0,
\]
satisfies $f_n(n\rho)=Q_n(\lambda)e^{\lambda n\rho}p_{n,\lambda}(n\rho)$,
so $n^{-1}\log f_n(n\rho)$ and the transform objective
$\lambda\rho+n^{-1}\log Q_n(\lambda)$ differ exactly by
$n^{-1}\log p_{n,\lambda}(n\rho)$.  This difference is negligible when
$n\rho$ is a typical value of the tilted law.  As in the proof of the
lower bound in Cram\'er's theorem \cite[Section~2.2]{DemboZeitouni}, we
make it typical by matching the mean: the finite-$n$ minimizer of the
transform objective is exactly the parameter for which $p_{n,\lambda}$
has mean $n\rho$.  Log-concavity guarantees that the density at this
mean is at least $1/(4en\rho)$.  Consequently, evaluating the profile at
$n\rho$ loses only a polynomial factor relative to the optimized
transform.  The following lemma supplies the required density bound.

\begin{lemma}[Value of a log-concave density at its mean]\label{lem:mean-density}
Let $p$ be a log-concave probability density on $[0,\infty)$, taken in its
upper-semicontinuous representative, with mean $m>0$.  Then
\begin{equation}\label{eq:mean-density}
p(m)\ge e^{-1}\lVert p\rVert_\infty\ge\frac1{4em}.
\end{equation}
\end{lemma}

When $m=n\rho$, taking logarithms in \eqref{eq:mean-density} gives
$\log p(n\rho)\ge-\log(4en\rho)$, which is negligible after division
by $n$.

\begin{proof}
Extend $p$ by zero to $\R$.  Its positive support is an interval, so this
extension remains log-concave and has mean $m$.
Fradelizi's barycenter inequality \cite[Theorem~4, p.~522]{Fradelizi}
states that an integrable log-concave function $f$ on $\R^d$ with
barycenter $g$ satisfies $f(g)\ge e^{-d}\sup f$.  In dimension one this
gives $\lVert p\rVert_\infty\le e\,p(m)$.
If $X$ has density $p$, Markov's inequality gives
\[
\frac12\le\Pr\{X\le2m\}
=\int_0^{2m}p(x)\dd x\le2m\lVert p\rVert_\infty.
\]
Thus $\lVert p\rVert_\infty\ge1/(4m)$, proving \eqref{eq:mean-density}.
\end{proof}

The second ingredient concerns optimization.  Pointwise convergence alone
does not control function values at a parameter that changes with $n$.
For finite convex functions, it strengthens to uniform convergence on
compact intervals, which will allow comparison of the minimizing values.

\begin{lemma}[Local uniform convergence of convex functions]
\label{lem:convex-local-uniform}
Let $G_n$ be finite convex functions on an open interval $I$.  If
$G_n\to G$ pointwise and $G$ is finite, then $G$ is convex and the
convergence is uniform on every compact subinterval of $I$.
\end{lemma}

\begin{proof}
Passing to the limit in
$G_n((1-t)x+ty)\le(1-t)G_n(x)+tG_n(y)$ shows that $G$ is convex.
Uniform convergence on compact subintervals is the one-dimensional case of
Rockafellar~\cite[Theorem~10.8, pp.~90--91]{Rockafellar}.
\end{proof}

\begin{proposition}[Exact-mean inversion]\label{prop:inversion}
Let $f_n\colon(0,\infty)\to(0,\infty)$ be nondecreasing and log-concave.  Put
\[
Q_n(\lambda)=\int_0^\infty e^{-\lambda c}f_n(c)\dd c,
\qquad
\Lambda_n(\lambda)=\frac1n\log Q_n(\lambda).
\]
Assume $Q_n(\lambda)<\infty$ for every $\lambda>0$.
Suppose $\Lambda_n\to\Lambda$ pointwise on $(0,\infty)$, with $\Lambda$
finite.  If, for some $\rho>0$, the function
$G_\rho(\lambda)=\lambda\rho+\Lambda(\lambda)$ is coercive, then
\begin{equation}\label{eq:inversion}
\lim_{n\to\infty}\frac1n\log f_n(n\rho)
=\inf_{\lambda>0}\{\lambda\rho+\Lambda(\lambda)\}.
\end{equation}
\end{proposition}

\begin{proof}
\emph{Construction of the minimizing parameter.}
Fix $\rho>0$ and write
\[
G_{n,\rho}(\lambda)=\lambda\rho+\Lambda_n(\lambda),
\qquad
p_{n,\lambda}(c)=Q_n(\lambda)^{-1}e^{-\lambda c}f_n(c).
\]
For any $a>0$, monotonicity gives
\[
Q_n(\lambda)\ge f_n(a)\int_a^\infty e^{-\lambda c}\dd c
=\frac{f_n(a)}\lambda e^{-\lambda a}.
\]
Taking $a=n\rho/2$ yields
\[
G_{n,\rho}(\lambda)
\ge\frac{\lambda\rho}{2}
+\frac{\log f_n(n\rho/2)-\log\lambda}{n}.
\]
For fixed $n$, the right-hand side tends to $+\infty$ both as
$\lambda\downarrow0$ (through $-\log\lambda/n$) and as $\lambda\to\infty$
(through $\lambda\rho/2$), so $G_{n,\rho}$ is coercive.

To differentiate $Q_n$, restrict $\lambda$ to a compact interval
$[\lambda_-,\lambda_+]\subset(0,\infty)$.
For $r=0,1,2$, there is a constant $C_r$ such that
\[
c^re^{-\lambda c}f_n(c)
\le C_re^{-\lambda_-c/2}f_n(c).
\]
The right-hand side is integrable because $Q_n(\lambda_-/2)<\infty$.
Dominated convergence gives
\[
Q_n'(\lambda)=-\int_0^\infty ce^{-\lambda c}f_n(c)\dd c,
\qquad
Q_n''(\lambda)=\int_0^\infty c^2e^{-\lambda c}f_n(c)\dd c.
\]
If $X$ has density $p_{n,\lambda}$, division by $Q_n$ identifies the
first two moments as
\[
m_n(\lambda)=\E_{p_{n,\lambda}}X=-\frac{Q_n'}{Q_n},
\qquad
\E_{p_{n,\lambda}}X^2=\frac{Q_n''}{Q_n}.
\]
Consequently,
\begin{align}
G_{n,\rho}'(\lambda)
&=\rho+\frac{Q_n'(\lambda)}{nQ_n(\lambda)}
=\rho-\frac{m_n(\lambda)}n,
\label{eq:tilted-mean}\\
G_{n,\rho}''(\lambda)
&=\frac1n\left(\frac{Q_n''}{Q_n}
-\left(\frac{Q_n'}{Q_n}\right)^2\right)
=\frac1n\Var_{p_{n,\lambda}}(X)>0.
\label{eq:tilted-variance}
\end{align}
The variance is positive because $p_{n,\lambda}$ is positive on
$(0,\infty)$ and therefore cannot be concentrated at one point.
Coercivity, continuity, and strict convexity give a unique interior
minimizer $\lambda_{n,\rho}$.  At that point the derivative in
\eqref{eq:tilted-mean} vanishes, so
\[
0=\rho-\frac{m_n(\lambda_{n,\rho})}{n},
\qquad m_n(\lambda_{n,\rho})=n\rho.
\]
Thus minimization selects precisely the density whose mean is the point
where we need to evaluate $f_n$.  This is the mean matching described at
the beginning of the section: after the change of measure, the radius
$n\rho$ is the mean of the tilted law, and the next step shows that the
tilted density there is not exponentially small.

\emph{Evaluation at the mean.}
The density $p_{n,\lambda}$ is log-concave, since
\[
\log p_{n,\lambda}(c)
=-\log Q_n(\lambda)-\lambda c+\log f_n(c)
\]
is the sum of a constant, an affine function, and a concave function of $c$.
Moreover, positivity and log-concavity make $f_n$ continuous on
$(0,\infty)$, so $p_{n,\lambda}$ agrees there with its
upper-semicontinuous representative.  Lemma~\ref{lem:mean-density} gives
\[
p_{n,\lambda_{n,\rho}}(n\rho)\ge\frac1{4en\rho}.
\]
For the lower bound, rearrange the definition of the density at its mean.
\[
f_n(n\rho)=Q_n(\lambda_{n,\rho})e^{\lambda_{n,\rho}n\rho}
p_{n,\lambda_{n,\rho}}(n\rho).
\]
Taking logarithms, dividing by $n$, and applying the density bound yields
\begin{align*}
\frac1n\log f_n(n\rho)
&=G_{n,\rho}(\lambda_{n,\rho})
+\frac1n\log p_{n,\lambda_{n,\rho}}(n\rho)\\
&\ge G_{n,\rho}(\lambda_{n,\rho})-\frac{\log(4en\rho)}n.
\end{align*}
For the upper bound, use monotonicity directly for any $\lambda>0$.
\[
Q_n(\lambda)\ge f_n(n\rho)\int_{n\rho}^\infty e^{-\lambda c}\dd c
=\frac{f_n(n\rho)}\lambda e^{-\lambda n\rho}.
\]
Rearranging and then setting $\lambda=\lambda_{n,\rho}$ gives the upper
inequality in the following pair of bounds.
\begin{equation}\label{eq:pointidentity}
G_{n,\rho}(\lambda_{n,\rho})-\frac{\log(4en\rho)}n
\le\frac1n\log f_n(n\rho)
\le G_{n,\rho}(\lambda_{n,\rho})+\frac{\log\lambda_{n,\rho}}n.
\end{equation}

\emph{Convergence of the minima.}
Equation~\eqref{eq:tilted-variance} shows that each $G_{n,\rho}$ is finite
and convex.  By Lemma~\ref{lem:convex-local-uniform}, $G_\rho$ is convex, hence
continuous, and the pointwise convergence $G_{n,\rho}\to G_\rho$ is uniform
on compact subintervals of $(0,\infty)$.  Together with coercivity, this
makes the minimizer set of $G_\rho$ a nonempty compact interval.  Choose $a,b$ outside that interval and fix any
minimizer $\lambda_\rho$, so that $0<a<\lambda_\rho<b$.  The gap
\[
\delta=\min\{G_\rho(a),G_\rho(b)\}-G_\rho(\lambda_\rho)
\]
is positive.  For all sufficiently large $n$, uniform convergence gives
\[
\min\{G_{n,\rho}(a),G_{n,\rho}(b)\}
>G_{n,\rho}(\lambda_\rho)+\delta/2.
\]
These endpoint inequalities confine the finite-$n$ minimizers to $[a,b]$.
Indeed, a convex function is nondecreasing to the right of a minimizer.
If $\lambda_{n,\rho}<a<\lambda_\rho$, this would give
$G_{n,\rho}(a)\le G_{n,\rho}(\lambda_\rho)$, contradicting the gap.
Similarly, a convex function is nonincreasing to the left of a minimizer,
so $\lambda_{n,\rho}>b>\lambda_\rho$ would give the contradictory
inequality at $b$.

Both global minima are therefore attained in the same interval $[a,b]$
for all sufficiently large $n$.  If two functions differ by at most
$\varepsilon$ everywhere on that interval, their minimum values differ
by at most $\varepsilon$.  This follows by evaluating each function at a
minimizer of the other.  Applying this observation gives
\[
\left|\min_{\lambda>0}G_{n,\rho}(\lambda)
-\min_{\lambda>0}G_\rho(\lambda)\right|
\le\sup_{\lambda\in[a,b]}
|G_{n,\rho}(\lambda)-G_\rho(\lambda)|\longrightarrow0.
\]
The confinement also bounds $|\log\lambda_{n,\rho}|$ independently of
large $n$.  Thus the upper error in \eqref{eq:pointidentity} is $O(1/n)$,
and the lower error is $O(\log n/n)$ for fixed $\rho$.  Equivalently,
\[
\frac1n\log f_n(n\rho)
=\min_{\lambda>0}G_{n,\rho}(\lambda)
+O\!\left(\frac{\log n}{n}\right).
\]
This estimate compares the fixed-level value with the finite-$n$
optimized transform; it does not assert an $O(\log n/n)$ rate of
convergence to the limiting profile.
Combining this estimate with convergence of the minima proves
\eqref{eq:inversion}.
\end{proof}

\begin{proof}[Proof of Theorem~\ref{thm:main}]
We apply Proposition~\ref{prop:inversion} by checking its profile,
transform, and minimization hypotheses in turn.
For all sufficiently large $n$, $1\le j_n\le n-1$.  Set
$f_n=f_{n,j_n}$, that is, $f_n(c)=V_{j_n}(K_n(c))$.
Lemma~\ref{lem:profile} gives positivity, monotonicity, polynomial growth,
and log-concavity of $f_n$.  Polynomial growth makes every $Q_n(\lambda)$
finite, so all finite-$n$ assumptions hold.

The Laplace transform of this $f_n$ is exactly $Q_{n,j_n}$, so
Proposition~\ref{prop:transform-rate} gives
\[
\Lambda_n(\lambda)=\frac1n\log Q_{n,j_n}(\lambda)
\longrightarrow
\Lambda(\lambda)=H(\alpha)-\frac{1-\alpha}{2}\log\pi+\PhiA(\lambda).
\]
Finally, Lemma~\ref{lem:coercive} supplies coercivity and strict convexity
of $\lambda\rho+\Lambda(\lambda)$ for every $\rho>0$, hence existence
and uniqueness of its minimizing parameter.
Proposition~\ref{prop:inversion} therefore yields
\[
\lim_{n\to\infty}\frac1n\log V_{j_n}(K_n(n\rho))
=H(\alpha)-\frac{1-\alpha}{2}\log\pi
+\inf_{\lambda>0}\{\lambda\rho+\PhiA(\lambda)\},
\]
which is \eqref{eq:mainrho}, including its specialization to $\rho=1$.
The inner maximizer is unique by Lemma~\ref{lem:inner} and the outer
minimizer by Lemma~\ref{lem:coercive}, and Lemma~\ref{lem:optimality}
identifies the optimal pair as the unique solution of
\eqref{eq-optimality-main}.
\end{proof}

We now prove Corollary~\ref{cor:profile}.  The inversion also covers
volume, the case $j=n$ excluded from Proposition~\ref{prop:product}.
Log-concavity in the order $j$, together with continuity at the endpoints,
then gives convergence uniform over all orders.  The interpolation
argument follows \cite[Section~5.3]{KabluchkoMarynychProchno}.

Formally, the right-hand side of \eqref{eq:mainrho} at $\alpha=1$ reduces
to $\rate(1,\rho)$, since $H(1)=0$, the $\log\pi$ term vanishes, and
\[
\sup_{\theta>0}\phi_1(\theta,\lambda)
=\sup_{\theta>0}\log Z(\theta,\lambda)
=\log Z(0,\lambda),
\]
because $Z$ is nonincreasing in $\theta$.  Part~(i)
is the exponential volume asymptotics of Orlicz balls in Gibbs form, and the
minimizer is the parameter of the maximum entropy principle of
\cite{KabluchkoProchno}.

The following variant of Lemma~\ref{lem:convex-local-uniform} controls
the endpoints of the order interval.

\begin{lemma}[Uniform convergence up to the endpoints]
\label{lem:convex-closed-uniform}
Let $G_n$ be finite convex functions on a compact interval $[a,b]$ that
converge pointwise on $[a,b]$ to a function $G$ that is finite and
continuous on $[a,b]$.  Then $G_n\to G$ uniformly on $[a,b]$.
\end{lemma}

\begin{proof}
By Lemma~\ref{lem:convex-local-uniform}, the convergence is uniform on
every compact subinterval of $(a,b)$.  Fix $\varepsilon>0$ and choose
$0<\delta<(b-a)/2$ with $|G(x)-G(a)|\le\varepsilon$ for
$x\in[a,a+2\delta]$.  For all sufficiently large $n$, $|G_n-G|\le\varepsilon$
at the points $a$ and $a+\delta$ and on the interval
$[a+\delta,a+2\delta]$.  Let $x\in[a,a+\delta]$.  Convexity gives the
upper bound
\[
G_n(x)\le\max\{G_n(a),G_n(a+\delta)\}\le G(a)+2\varepsilon\le G(x)+3\varepsilon.
\]
For the lower bound, $a+\delta$ is the midpoint of $x$ and
$y=2(a+\delta)-x\in[a+\delta,a+2\delta]$, so convexity gives
\[
G_n(x)\ge2G_n(a+\delta)-G_n(y)\ge2G(a+\delta)-G(y)-3\varepsilon
\ge G(a)-6\varepsilon\ge G(x)-7\varepsilon.
\]
Thus $|G_n-G|\le7\varepsilon$ on $[a,a+\delta]$ for all large $n$.  The
same argument applies at $b$, and $[a+\delta,b-\delta]$ is covered by
the local uniform convergence.
\end{proof}

\begin{proof}[Proof of Corollary~\ref{cor:profile}]
(i) Let $f_n(c)=\Vol(K_n(c))$.  By Lemma~\ref{lem:profile} with $j=n$,
$f_n$ is positive, nondecreasing, log-concave, and of polynomial growth.
Tonelli gives, for $\lambda>0$,
\[
Q_n(\lambda)=\int_0^\infty e^{-\lambda c}\Vol(K_n(c))\dd c
=\int_{\R^n}\int_{F_n(x)}^\infty e^{-\lambda c}\dd c\dd x
=\frac1\lambda\int_{\R^n}e^{-\lambda F_n(x)}\dd x
=\frac{Z(0,\lambda)^n}{\lambda},
\]
so $\Lambda_n(\lambda)=\log Z(0,\lambda)-n^{-1}\log\lambda$ converges to
$\log Z(0,\lambda)$, which is finite by Lemma~\ref{lem:endpoints}.
Monotone convergence gives $Z(0,\lambda)\uparrow\infty$ as
$\lambda\downarrow0$, and restricting the integral to $|s|\le\varepsilon$
gives
$\lambda\rho+\log Z(0,\lambda)\ge\lambda(\rho-M(\varepsilon))+\log(2\varepsilon)$,
which tends to $\infty$ as $\lambda\to\infty$ once $M(\varepsilon)<\rho$.
Thus $\lambda\rho+\log Z(0,\lambda)$ is coercive, and
Proposition~\ref{prop:inversion} yields the limit.  The bounds in the
proof of Lemma~\ref{lem:endpoints} hold for $\theta=0$, so, as in the
proof of Lemma~\ref{lem:coercive},
$\partial_\lambda\log Z(0,\lambda)=-\E_{\mu_{Z,0,\lambda}}[M]$ and
$\partial_\lambda^2\log Z(0,\lambda)=\Var_{\mu_{Z,0,\lambda}}(M)>0$, which
give uniqueness and the characterization of the minimizer.

(ii) The sequence $(V_r(K))_{r=0}^n$ of intrinsic volumes of a convex body
$K\subseteq\R^n$ is log-concave, that is,
$V_r(K)^2\ge V_{r-1}(K)V_{r+1}(K)$ for $1\le r\le n-1$, a consequence of
the Alexandrov--Fenchel inequality due to McMullen~\cite{McMullen}, see
also \cite[Remark~2.20]{KabluchkoMarynychProchno}.  Hence, for each $n$, the sequence
$j\mapsto n^{-1}\log V_j(K_n(n\rho))$, $j=0,\ldots,n$, has nonincreasing
increments, and its piecewise linear interpolation $\zeta_n$ on $[0,1]$,
with nodes at $j/n$, is a finite concave function.

We claim that $\zeta_n\to\rate(\cdot,\rho)$ pointwise on $[0,1]$.  At the
endpoints, $\zeta_n(0)=n^{-1}\log V_0(K_n(n\rho))=0$ for every $n$, and
$\zeta_n(1)=n^{-1}\log\Vol(K_n(n\rho))\to\rate(1,\rho)$ by part~(i).  Fix
$x\in(0,1)$ and put $j_n=\lfloor nx\rfloor$.  For all sufficiently large
$n$, both $j_n$ and $j_n+1$ lie in $\{0,\ldots,n-1\}$, and $j_n/n\to x$,
$(j_n+1)/n\to x$.  By Theorem~\ref{thm:main}, the values of $\zeta_n$ at
these two nodes converge to $\rate(x,\rho)$, and $\zeta_n(x)$, a convex
combination of them, converges to the same limit.

Passing to the limit in the concavity inequality for $\zeta_n$ shows that
the extended profile is concave on $[0,1]$, and the chord between the
nodes $0$ and $1$ gives $\rate(\alpha,\rho)\ge\alpha\,\rate(1,\rho)$.  Hence
$\liminf_{\alpha\uparrow1}\rate(\alpha,\rho)\ge\rate(1,\rho)$ and
$\liminf_{\alpha\downarrow0}\rate(\alpha,\rho)\ge0$.  For the reverse
inequalities, Lemma~\ref{lem:endpoints} gives
$\sqrt\theta\,B(\theta,\lambda)\le\sqrt\pi$ and
$Z(\theta,\lambda)\le Z(0,\lambda)$, so
$\PhiA(\lambda)\le\alpha\log Z(0,\lambda)+\frac{1-\alpha}{2}\log\pi$ and,
for every fixed $\lambda_0>0$,
\[
\rate(\alpha,\rho)\le H(\alpha)+\lambda_0\rho+\alpha\log Z(0,\lambda_0).
\]
Letting $\alpha\uparrow1$ and then taking the infimum over $\lambda_0$
gives $\limsup_{\alpha\uparrow1}\rate(\alpha,\rho)\le\rate(1,\rho)$, and
letting $\alpha\downarrow0$ and then $\lambda_0\downarrow0$ gives
$\limsup_{\alpha\downarrow0}\rate(\alpha,\rho)\le0$.  Continuity on
$(0,1)$ follows from concavity, so the extended profile is continuous on
$[0,1]$.  The functions $-\zeta_n$ are convex on $[0,1]$ and converge
pointwise there to $-\rate(\cdot,\rho)$, so
Lemma~\ref{lem:convex-closed-uniform} makes the convergence uniform on
$[0,1]$.  Since $\zeta_n(j/n)=n^{-1}\log V_j(K_n(n\rho))$, this is
\eqref{eq:full-uniform}.  If $j_n/n\to\alpha$, then \eqref{eq:full-uniform}
and continuity of the profile give
$n^{-1}\log V_{j_n}(K_n(n\rho))-\rate(j_n/n,\rho)\to0$ and
$\rate(j_n/n,\rho)\to\rate(\alpha,\rho)$.  Bounded sequences are included
as a special case of $j_n=o(n)$, and
$j_n=\lfloor\alpha n\rfloor$ is the last assertion.
\end{proof}

We close this section with the duality between the radius profile and
the transform rate announced in the introduction.

\begin{remark}[Dependence on the radius]\label{rem:radius}
Fix $\alpha\in(0,1)$.  As an infimum of increasing affine functions of
$\rho$, the profile $\rho\mapsto\rate(\alpha,\rho)$ is concave,
nondecreasing, and continuous on $(0,\infty)$.  More precisely,
\[
\rate(\alpha,\rho)=\inf_{\lambda>0}
\{\lambda\rho+\Lambda(\lambda)\},
\qquad
\Lambda(\lambda)=\sup_{\rho>0}
\{\rate(\alpha,\rho)-\lambda\rho\}.
\]
The first identity is \eqref{eq:mainrho}.  For the reverse identity, the
first one gives
$\rate(\alpha,\rho)-\lambda\rho\le\Lambda(\lambda)$ for every $\rho>0$.
Conversely, fix $\lambda>0$ and put
$\rho_\lambda=-\Lambda'(\lambda)>0$.  Convexity makes $\lambda$ the
minimizer of $t\rho_\lambda+\Lambda(t)$, so equality holds at
$\rho=\rho_\lambda$.
By Lemmas~\ref{lem:coercive} and~\ref{lem:optimality}, the map
$-\Lambda'$ is continuous and strictly decreasing onto $(0,\infty)$,
and $-\Lambda'(\lambda_\rho)=\rho$.  Its inverse
$\rho\mapsto\lambda_\rho$ is therefore continuous.  Evaluating each of
the variational problems at the other's minimizer gives, for
$h\ne0$ with $\rho+h>0$,
\[
\lambda_{\rho+h}h
\le\rate(\alpha,\rho+h)-\rate(\alpha,\rho)
\le\lambda_\rho h.
\]
Dividing by $h$, reversing the inequalities when $h<0$, and letting
$h\to0$ yields $\partial_\rho\rate(\alpha,\rho)=\lambda_\rho$.
Thus $\lambda_\rho$ is the unique supergradient at $\rho$.

These identities have a thermodynamic reading, which we record only as
an interpretation, in the terminology of \cite[Section~5.2]{Touchette}.
Regard the radius $c$ as an energy and $\lambda$ as an inverse
temperature.  For $j=n$, Tonelli's theorem gives
$Q_{n,n}(\lambda)=\lambda^{-1}\int_{\R^n}e^{-\lambda F_n(x)}\dd x$, so
$c\mapsto\Vol(K_n(c))$ is the integrated density of states of the energy
$F_n$ and $Q_{n,n}$ is, up to the factor $\lambda^{-1}$, its canonical
partition function \cite[(146)]{Touchette}.  For $j<n$, the profile
$c\mapsto V_j(K_n(c))$ plays the same role.  Accordingly,
$\rho\mapsto\rate(\alpha,\rho)$ is the corresponding microcanonical
entropy $s(u)$ of \cite[Section~5.2.1]{Touchette}, with $u=\rho$, and
$\Lambda$ is the limiting log-partition function per coordinate, so that
$-\Lambda$ is the free energy function $\varphi(\beta)$ of
\cite[(145)]{Touchette} at $\beta=\lambda$, that is, $\lambda$ times the
thermodynamic free energy.  The two identities above are then, in this
order, the Legendre--Fenchel relations
\cite[Section~5.2.3, (149) and~(148)]{Touchette} between entropy and
free energy.  The first one expresses the equivalence of the
microcanonical and canonical ensembles at the thermodynamic level, which
holds when the entropy is concave \cite[Section~5.5]{Touchette}, as
$\rho\mapsto\rate(\alpha,\rho)$ is here.
\end{remark}

\section{Consistency with known cases}\label{sec:known-cases}

\emph{Quadratic case.}
For \(M(s)=s^2\),
\[
Z(\theta,\lambda)=\sqrt{\frac{\pi}{\lambda+4\theta}},
\qquad
B(\theta,\lambda)=2Z(\theta,\lambda).
\]
The inner stationarity equation gives
\(\theta_*(\lambda)=(1-\alpha)\lambda/(4\alpha)\), and the outer
minimum is attained at \(\lambda_\rho=\alpha/(2\rho)\).  Substitution in
\eqref{eq:mainrho} yields
\begin{equation}\label{eq:quadratic}
\rate(\alpha,\rho)
=H(\alpha)+\frac{\alpha}{2}\log(2\pi e\rho)
+\frac{1-\alpha}{2}\log(1-\alpha).
\end{equation}
This also follows independently from the exact identity
\[
V_j(RB_2^n)=R^j\binom nj\frac{\kappa_n}{\kappa_{n-j}}.
\]
Indeed, with \(R=\sqrt{n\rho}\), \(j=j_n\), and
\(\alpha_n=j_n/n\), Stirling's formula gives
\[
\frac1n\log V_{j_n}(\sqrt{n\rho}\,B_2^n)
=H(\alpha_n)+\frac{\alpha_n}{2}\log(2\pi e\rho)
+\frac{1-\alpha_n}{2}\log(1-\alpha_n)+o(1),
\]
which converges to \eqref{eq:quadratic}.  At the endpoint,
Corollary~\ref{cor:profile} gives
$\rate(1,\rho)=\inf_{\lambda>0}\{\lambda\rho+\tfrac12\log(\pi/\lambda)\}
=\tfrac12\log(2\pi e\rho)$, agreeing with the limit of
\eqref{eq:quadratic} as $\alpha\uparrow1$.  Equivalently,
\[
\lim_{n\to\infty}\Vol(\sqrt{n\rho}\,B_2^n)^{1/n}
=\sqrt{2\pi e\rho}.
\]

\emph{The \(\ell_p\) specialization.}
Let \(M(s)=|s|^p\), \(p>1\).  For \(\alpha\in(0,1)\) and \(u>0\), put
\begin{align*}
I_p(u)&=\int_{\R}e^{-|x|^p-u|x|^{2p-2}}\dd x,\\
J_p(u)&=\int_{\R}|x|^{p-2}e^{-|x|^p-u|x|^{2p-2}}\dd x,\\
K_p(u)&=\int_{\R}|x|^{2p-2}e^{-|x|^p-u|x|^{2p-2}}\dd x,\\
\Psi_{p,\alpha}(u)
&=\frac{1-\alpha}{2}\log u
+\alpha\log I_p(u)+(1-\alpha)\log J_p(u).
\end{align*}
The singularity in \(J_p\) is integrable also for \(1<p<2\).  The
notation \(I_p\), \(J_p\), \(K_p\), and \(\Psi_{p,\alpha}\) is that of
\cite{KabluchkoMarynychProchno} and is unrelated to the profile
\(\rate\), the integrals \(I_n(\lambda)\), the interval \(J\) of
Lemma~\ref{lem:endpoints}, and the bodies \(K_n(c)\).
With
\[
s=\lambda^{-1/p}x,
\qquad
u=p^2\lambda^{-(2p-2)/p}\theta,
\]
we have
\begin{align*}
Z(\theta,\lambda)&=\lambda^{-1/p}I_p(u),&
B(\theta,\lambda)&=p(p-1)\lambda^{-(p-1)/p}J_p(u),\\
A(\theta,\lambda)&=p^2\lambda^{-(2p-1)/p}K_p(u).
\end{align*}
Consequently,
\[
\PhiA(\lambda)
=-\frac{\alpha}{p}\log\lambda
+(1-\alpha)\log(p-1)
+\sup_{u>0}\Psi_{p,\alpha}(u).
\]
The outer minimum is attained at
\(\lambda_\rho=\alpha/(p\rho)\).  Hence, writing
\(g_p(\alpha)=\rate(\alpha,1)\),
\begin{equation}\label{eq:gp}
\begin{aligned}
g_p(\alpha)
={}&H(\alpha)
+\frac{\alpha}{p}\left(1-\log\frac{\alpha}{p}\right)
+(1-\alpha)\log(p-1)\\
&-\frac{1-\alpha}{2}\log\pi
+\sup_{u>0}\Psi_{p,\alpha}(u),
\end{aligned}
\end{equation}
and
\[
\rate(\alpha,\rho)=g_p(\alpha)+\frac{\alpha}{p}\log\rho.
\]
Here \(\Psi_{p,\alpha}\) is the phase function of
\cite[Proposition~2.12]{KabluchkoMarynychProchno}, and for
\(\alpha\in(0,1)\) the right-hand side of \eqref{eq:gp} is the exponential
profile \(g_p(\alpha)\) of the normalized ball \(n^{1/p}B_p^n=K_n(n)\) in
\cite[Theorem~2.17]{KabluchkoMarynychProchno}.  At \(\alpha=0\) and
\(\alpha=1\) both profiles are the limits of the same sequences
\(n^{-1}\log V_0(K_n(n))=0\) and \(n^{-1}\log\Vol(K_n(n))\), by
Corollary~\ref{cor:profile} on our side, so the two profiles coincide on
all of \([0,1]\), and for \(p=2\) the formula reduces to
\eqref{eq:quadratic}.  Homogeneity gives, for every sequence \(j_n\),
\[
\frac1n\log V_{j_n}(B_p^n)
=\frac1n\log V_{j_n}(K_n(n))
-\frac{j_n}{pn}\log n.
\]
When \(j_n/n\to\alpha>0\), the correction diverges to \(-\infty\), which
explains the normalization \(n^{1/p}B_p^n\).
The same substitution in
Proposition~\ref{prop:product} recovers
\cite[Theorem~2.1]{KabluchkoMarynychProchno}.  Indeed, homogeneity gives
\[
Q_{n,j}(\lambda)
=\Gamma\Bigl(1+\frac jp\Bigr)\lambda^{-1-j/p}V_j(B_p^n),
\]
while the transformed right-hand side of \eqref{eq:product} has the same
factor \(\lambda^{-1-j/p}\).  Canceling it gives
\[
V_j(B_p^n)
=\frac{n\binom{n-1}{j}p(p-1)^{n-1-j}}
{2\pi^{(n-j)/2}\Gamma(1+j/p)}
\int_0^\infty u^{(n-j)/2-1}
 I_p(u)^jJ_p(u)^{n-1-j}K_p(u)\dd u,
\]
which is exactly \cite[Theorem~2.1]{KabluchkoMarynychProchno}.

\section{Concluding remarks}\label{sec-open}

\emph{Generalized Orlicz balls.}
For coordinate-dependent functions $M_1,\ldots,M_n$ as in
\cite{JohnstonProchno}, the curvature, coarea, and Gamma steps of
Section~\ref{sec:factorization} go through, and the radius Laplace
transform of $V_j$ is again an integral in $\theta$ of a sum over
coordinate subsets of products of one-dimensional integrals.  The sum no
longer collapses to a binomial coefficient, and identifying its
exponential rate is open.

\emph{Curvature measures.}
Remark~\ref{rem:optimality} identifies the mixture
$\alpha\mu_{Z,\theta,\lambda}+(1-\alpha)\mu_{B,\theta,\lambda}$ at the
optimal parameters as the natural candidate for the limit law of one
coordinate under the normalized curvature measure
$\mathcal C_{j_n}(K_n(n\rho),\cdot)/V_{j_n}(K_n(n\rho))$, as in
\cite[Theorem~3.1]{KabluchkoMarynychProchno} for $\ell_p$ balls, and
Proposition~\ref{prop:product} is the natural starting point for such a
Maxwell--Poincar\'e--Borel law.

\emph{Regularity and symmetry of $M$.}
Remark~\ref{rem-assumptions} isolates where \textnormal{(M2)} is used,
and Section~\ref{sec:defining} explains how evenness can be replaced by
the requirement that $0$ minimizes $M$.  Extending the curvature-measure
step to functions with locally absolutely continuous derivative would
remove the $C^2$ requirement.

\emph{Rates of convergence.}
Proposition~\ref{prop:inversion} compares
$n^{-1}\log V_{j_n}(K_n(n\rho))$ with the minimum of the finite-$n$
transform objective up to $O(\log n/n)$.  A rate of convergence to the
limiting profile would in addition require a quantitative version of
Proposition~\ref{prop:transform-rate}, that is, a Laplace-method
expansion of the $\theta$-integral in \eqref{eq:product} with errors
uniform in $\lambda$ on compact sets, which we have not attempted.

\section*{Declaration of AI use}

The author used the large language model assistants ChatGPT (OpenAI) and
Claude (Anthropic) in conducting this research and preparing the
manuscript.  They were used to suggest proof ideas and derivations, to
check the arguments for errors and gaps, and to draft and revise parts
of the text.  The author checked all results and arguments and takes
full responsibility for the content of the paper.

\end{document}